\documentclass{CSML}
\pdfoutput=1

\usepackage{lastpage}
\newcommand{\dOi}{14(2:17)2018}
\lmcsheading{}{1--\pageref{LastPage}}{}{}%
{Mar.~20, 2017}{Jun.~22, 2018}{}


\usepackage{hyperref}
\hypersetup{hidelinks}

\newcommand {\m}{\mu}

\newcommand{\ma}{\mathcal}

\newcommand{\lmr}{\l \mu \rho}
\newcommand {\lmp}{\l \mu \mu'}

\newcommand{\lmrt}{\l \mu \rho \th}
\newcommand{\bmrt}{\b \mu \rho \th}

\def\a{\alpha}
\def \b{\beta}

\def\g{\gamma}
\def\l{\lambda}

\def\r{\rho}

\def\th{\theta}
\def\lm{\l \mu}
\def\lmp{\l \mu \mu'}

\def\up{\overline}
\def\ora{\overrightarrow}

\def\tx{\textrm}
\def\ra{\rightarrow}

\def\tw{\textrm{tower}}
\def\twe{\emph{tower}}

\def\lan{\langle}
\def\ran{\rangle}

\def\si{\sigma}
\def\r{\rho}

\def\fl{\mapsto}
\def\up{\overline}
\def\tx{\textrm}
\def\v{\vdash}
\def\G{\Gamma}

\def\D{\triangle}
\def\F{\displaystyle\frac}
\def\bdf{\begin{defi}}
\def\edf{\end{defi}}
\def\be{\begin{enumerate}}
\def\ee{\end{enumerate}}
\def\bd{\begin{description}}
\def\ed{\end{description}}
\def\ite{\begin{itemize}}
\def\ete{\end{itemize}}
\def\bp{\begin{proof}}
\def\ep{\end{proof}}
\def\bl{\begin{lem}}
\def\el{\end{lem}}
\def\bt{\begin{thm}}
\def\et{\end{thm}}
\def\bpt{\begin{prooftree}}
\def\ept{\end{prooftree}}
\def\bdf{\begin{defi}}
\def\edf{\end{defi}}

\begin{document}

\title[An estimation for the lengths of reduction sequences of the $\lambda\mu\rho\theta$-calculus]
{An estimation for the lengths of reduction sequences of the $\lambda\mu\rho\theta$-calculus}

\author[P.~Batty\'anyi]{P\'eter Batty\'anyi}	
\address{Department of Computer Science, Faculty of Informatics,
University of Debrecen, Kassai \'ut 26, 4028 Debrecen, Hungary}	
\email{battyanyi.peter@inf.unideb.hu}  

\author[K.~Nour]{Karim Nour}	
\address{Univ. Grenoble Alpes,
Univ. Savoie Mont Blanc,
CNRS, LAMA, 73000 Chamb\'ery, France}	
\email{karim.nour@univ-smb.fr}  

\begin{abstract}
Since it was realized that the Curry-Howard isomorphism can be extended to the case of classical logic as well,
several calculi have appeared as candidates for the encodings of proofs in classical logic. One of the most
extensively studied among them is the $\lm$-calculus of Parigot \cite{Par1}. In this paper, based on the result of Xi presented for the $\l$-calculus
\cite{Xi}, we give an upper bound for the lengths of the reduction sequences in the $\lm$-calculus extended with the
$\rho$- and $\theta$-rules. Surprisingly, our results show that the new terms and the new rules do not add to the computational complexity
of the calculus despite the fact that $\m$-abstraction is able to consume an unbounded number of arguments by virtue of the $\m$-rule.
\end{abstract}

\maketitle

\section{Introduction}

\subsection{The Curry-Howard isomorphism for classical logic}

In the early nineties it was realized that the Curry-Howard isomorphism can be extended to the case of classical logic
\cite{Gri,Mur}. Since then several calculi have appeared aiming to give an encoding of proofs formulated
either in classical natural deduction or in classical sequent calculus \cite{Ber-Bar,Cur-Her,Par1,Reh}.

A noteworthy example of a calculus establishing a correspondence between terms and natural deduction proofs is the $\lm$-calculus presented
by Parigot \cite{Par5}, which stands very close in nature to the $\l$-calculus itself.
Besides the usual variables a new type of variables is introduced, the so-called classical- or $\m$-variables.
The calculus enriched with the $\m$-variables is capable of representing proofs in classical natural deduction
by terms via the Curry-Howard isomorphism. The reduction rules corresponding to the new $\l\m$-terms are
defined in \cite{Par1}. In addition, further simplification rules, for example the $\rho$- and $\theta$-rules,
and the symmetric counterpart of the $\m$-rule,
which is the $\m'$-rule, were defined for the $\l\m$-calculus \cite{Par1,Par2}. The motivation
for introducing the $\m'$-rule, and the simplification rules $\rho$ and $\theta$, was the following.
In the typed $\l$-calculus we are able to define integers by Church's numerals and other data types in the usual manner \cite{Kri1}.
For the Church numerals and the data types the unicity of representation of data holds. This means that, if we talk
about the Church numerals only, every term of type $N$, where $N$ is a type of a Church numeral, is $\b$-equal to a Church numeral.
This is no more true for the $\lm$-calculus: we can find normal terms of type $N$ that are not Church numerals. The problem is
resolved by introducing a symmetric equivalent of the $\m$-rule and some more reduction rules, namely the $\rho$- and the $\th$-rules
\cite{Par2}. We should remark that the price for adding more rules to the calculus other than the $\b$- and $\m$-rules was the disappearance of the usual proof
theoretical properties, like confluence, in case of the $\lm\m'$-calculus. Even strong normalization is lost, when we
consider the symmetric $\lm$-calculus together with the $\rho$-rule \cite{Batt}. Parigot has already showed in \cite{Par3} that
the $\lm$-calculus, i.e. when we consider the $\b$- and the $\mu$-rules only, is strongly normalizing: his proof was based
 on the Tait-Girard reducibility method \cite{Tai}. An arithmetical proof of the same result was presented by David and Nour \cite{Dav-Nou1}.

\subsection{The work of Xi}

Prior to presenting the work of Xi, we give an example which shows that the length of a
reduction sequence can be exponential in terms of the complexity of a term even in the case of the simply typed $\l$-calculus.
We define a sequence of simple types with recursion. Let $N_0 = (X \ra X) \ra (X \ra X)$, where $X$ is a type variable and,
for every $i \in \mathbb{N}$, $N_{i+1} = (N_i \ra N_i) \ra (N_i \ra N_i)$. For every $n \in \mathbb{N}$ we denote by $\underline{n}$ the
$n$-th Church numeral. Let $n_1,...,n_m \in \mathbb{N}$ and $P =(\underline{n_m} \;(... \; ( \underline{n_{3}} \;\;
(\underline{n_{2}}\; \; \underline{n_1})) \; ...))$.
It is easy to check that $P$ is of type $N_{m-1} \ra N_{m-1}$ and reduces to the Church numeral
$\underline{{n_m}^{.^{.^{.^{{\scriptstyle n_{2}}^{ {{\scriptstyle n_{1}}}}}}}} }$
in ${n_m}^{.^{.^{.^{{\scriptstyle n_{2}}^{ {{\scriptstyle n_{1}}}}}}}}$ steps of $\b$-reductions.

In his paper \cite{Xi}, Xi obtains an upper bound for the lengths of reduction sequences of the simply typed $\l$-calculus.
First of all, he finds a bound for the leftmost reduction sequences of the $\l $I-calculus. Since any reduction sequence of
a $\l $I-term is at most as long as the leftmost one, he has immediately found a bound for the $\l$I-calculus. Next, he maps the
set of $\l$-terms into the set of $\l$I-terms such that, for any reduction sequence of a $\l$-term, he can find a reduction
sequence of the corresponding $\l$I-term which is at least as long. As our starting example shows, it is inevitable that
this bound will be exponential in relation to the complexity (the number of symbols) of the term and/or the rank (the number of arrows)
of the type of the term.
In our treatment we chose to develop Xi's method further, since, when aiming to find the terms with longest reduction sequences,
$\l\m $I-terms appeared to be promising candidates and, among their reduction sequences, the standard ones turned out to be the ones
with longest reduction paths.

An improvement of Xi's method appeared in the work of Asperti and L\'evy \cite{AL}. They gave a refinement of Xi's result
which proved to be a considerable strengthening of the bound especially in the cases when the computation leads to normal forms
of a sufficiently simple structure (e.g. a variable or booleans). They showed that in these cases the length of the longest reduction sequence
is at most a factorial of that of the shortest one.

\subsection{The motivations of our work}

First of all, we note that the motivation for the introduction of the $\l\m$-calculus was not the enhancement of the expressive power of
the $\l$-calculus but rather the need for representing program constructs, like $exit$, $call/CC$, which were missing from the simply
typed $\l$-calculus. The power of the new calculus stems from the fact that a $\m$-abstraction can consume any number of
arguments through subsequent reductions.

Interestingly, our results show that the upper bound for the number of reduction steps of  a term in the $\l\m$-calculus stays close to
the expression giving an upper bound for the $\l $-calculus. One expects that this bound should increase in comparison with the $\l$-calculus,
since the $\l\m$-calculus properly contains the $\l$-calculus.

Intuitively, one should try to simulate the $\l\m$-calculus with the help of the $\l$-calculus and hence obtain a bound
for the lengths of the reduction sequences. This type of simulation has two defects: first of all, the upper bound for the lengths
would grow considerably and, secondly, additional difficulties would come up when we want to simulate the simplification rules
(for a detailed explanation see Section \ref{final:section}).
Hence, in spite of the expected difficulties with respect to the handling of critical pairs and, consequently, with relation
to the choice of the definition of standard reduction sequences, we have voted for the adaptation of Xi's method.

There are several works concerning the standardization of the $\l\m$-calculus \cite{Batt,Dav-Nou2,Py,Sau}.
David and Nour \cite{Dav-Nou2} consider standardization of the $\l\m$-calculus without simplification rules, while Py \cite{Py}
chooses simplification rules by which the confluence is retained in the extended calculi. This eliminates the difficulties imposed
on the treatment of the critical pairs. A different aspect is that of Saurin \cite{Sau}. He works with essentially the same calculus:
basically, he considers the rules $\b$, $\m$, $\rho$ and $\theta$, which are sufficient to obtain confluency in the calculus.
In his calculus, he also applies some other- local and global- simplification rules besides the already mentioned ones.
In spite of the complexity of this calculus, Saurin succeeds in defining the notion of a standard reduction sequence and he
also obtains a standardization theorem. His definition follows the traditional way: he uses the notion of residuals
(one cannot reduce the residual of a redex lying to the left of a redex reduced). This results in a rather technical proof.
However, none of the works mentioned contain estimations for the lengths of
standard reduction sequences. Instead of the traditional way, we chose to define the notion of a standard reduction sequence
following the style of David \cite{Dav-Nou2} so that the proofs can be carried out by recursion on the lengths of the reduction sequences.

\subsection{Outline of the present work}

In this paper, following the reasoning of Xi \cite{Xi} for the simply typed $\l$-calculus, we present an upper bound for the lengths of the
reduction sequences in the $\l\m\rho\theta$-calculus in terms of the complexity and the rank of the term $M$, where
the rank of $M$ is the maximum of its redex ranks and the complexity of a term is the number of symbols in $M$. We base our treatment
on \cite{Batt}. First we prove a standardization theorem for the $\l\m$-calculus with the additional assertion that, in the case of the
$\l\m $I-calculus the length of a reduction sequence is majorized by that of its standardization, which is, in turn, bounded from above
 by a certain measure defined in the paper.
In addition, we show that, if $M$ is a $\l\m$I-term, then a standard reduction sequence leading to the normal form of $M$
is the leftmost reduction sequence, which is necessarily unique. This fits our intuition, as a matter of fact. Thus it makes no difference
 in which way we are able to find the standard reduction sequence leading to the normal form of $M$: it is by all means the longest reduction
sequence normalizing $M$.
Thus, as our first task, we find an appropriate, normalizing reduction sequence for $\l\m$I-terms such that its standardization provides
us with a measure which is a super-exponential number theoretic function and an upper bound for the length of the standard reduction sequence.
This is accomplished in Section \ref{section:3}. Hence, our strategy for a general term $M$ is to define a translation $[\![M]\!]_k$ of $M$ into the $\l\m$I-calculus
such that the longest reduction sequence of $M$ is not longer than the longest reduction sequence of $[\![M]\!]_k$. Since we have already obtained
a bound for the reduction sequences of $[\![M]\!]_k$, we also have one for those of $M$.

\subsection{Some difficulties and our proposed solutions}

The new redexes, and especially the critical pairs, impose additional difficulties: we reformulated the notion of a standard reduction
sequence in the $\lm\rho\theta$-calculus in line with the definition of standard $\beta$-reduction sequences according to David \cite{Dav}.
The resulting notion proved to be vastly different in appearance from that for the $\l$-calculus. With the presence
of the new rules overlapping redexes can occur: performing a $\th$-redex can make a $\m$-redex to vanish and executing a $\th$-redex can make
a $\rho$-redex to disappear and vice versa. Our definition of a standard reduction sequence excludes these situations: overlapping $\th$- and
$\mu$-redexes are always considered $\mu$-redexes. Likewise, in the case overlapping $\th$- and $\rho$-redexes, a $\th$-redex in not allowed to
destroy the containing $\rho$-redex and vice versa. We show that our choice is appropriate: we can always find standard reduction sequences
respecting these constraints. Due to the presence of other reduction rules, finding the bound and proving that the lengths of the
reduction sequences obey that bound is more difficult even in the case of $\lm$I-terms. Instead of treating as candidates every possible
$\l \m$I-sequences of reductions for the estimation like Xi does for the $\l$I-calculus, we compute the bound only by starting from a special
kind of $\l\m$I-reduction sequence which we call good $k$-normalization sequence.
We evaluate the bound for the general case by assigning a $\lm$I-term $[\![M]\!]_k$ with a certain $k$ to the $\lm$-term $M$ such
that the length of the longest reduction paths of $[\![M]\!]_k$ is greater than or equal to that of $M$. Concerning the general case,
our transformation of $M$ into $[\![M]\!]_k$ is in fact a correction of Xi's argument \cite{Xi}, the original idea of Xi contained
a slight impreciseness. Namely, when $M$ is a $\l$-abstraction, we have to apply case distinction deciding whether $M$ is the left
hand side of a $\b$-redex or not. Otherwise the size of the corresponding $\l \m$I-term, $[\![M]\!]_k$, could not be estimated correctly.

As a general remark, we can observe that the bound obtained for the $\l\m$-calculus is exactly the same as the one obtained by Xi \cite{Xi}
for the $\l$-calculus, except for the fact that the ranks of the redexes and hence the measures of the reduction sequences are generalizations
of the corresponding notions for the $\l$-calculus. Perhaps, this surprising result could be interpreted as an informal statement saying that
the computational complexity of the $\lm$-calculus has not been increased by introducing the classical variables in the $\l$-calculus.

\section{The \texorpdfstring{$\lmrt$}{lambda-mu-rho-theta}-calculus}

\subsection{The syntax of calculus}

The $\lmrt$-calculus was introduced by Parigot \cite{Par1}.
Instead of his original calculus, we use a modified version owing to de Groote
\cite{de Gro1}: we apply the term formation rules in a more flexible way, that is,
we do not assume that a $\m$-abstraction $(\m\a. M)$ must be of the form $(\m\a. ([\b]M_1))$ for some $\b$ and $M_1$.
In what follows, we give the appropriate definitions.

\begin{defi}[Terms] \label{int:termsdegroote}
\hfill\begin{enumerate}
\item There are two kinds of variables : the set of $\l$-variables $\mathcal{V} = \{ x,y,z,\dots\}$ and the set of $\mu$-variables
$\mathcal{W} = \{\a,\b,\g,\dots\}$. The set of terms is denoted by $\mathcal{T}$ and the term formation rules are:
\[\mathcal{T} \; := \; \mathcal{V} \; \mid \; ({\l}\mathcal{V}.\mathcal{T}) \;
\mid \; (\m \mathcal{W}. \mathcal{T}) \; \mid \; ([\mathcal{W}] \mathcal{T}) \;
\mid \; (\mathcal{T}\;\mathcal{T}).\]

\item The complexity of a term is defined as follows:
\begin{itemize}
\item $comp(x)=1$,
\item $comp([\a] M)= comp(\l x.M) = comp(\mu \a. M) = comp(M)+1$,
\item $comp(M\,N)=comp(M)+comp(N)$.
\end{itemize}
\item As usual we denote by $Fv(M)$ the set of variables ooccuring free in the term $M$.
\item Let $M$ and $N$ be terms. We write $N \leq M$ if $N$ is a subterm of $M$ and $N<M$, if $N\leq M$ and $N\neq M$.
\end{enumerate}
\end{defi}

In brief, the complexity of a term is the number of symbols in the term. By the formation of terms we apply the usual stipulations:
the scope of the $\l$- and $\m$-abstractions extend to the right as far as possible, moreover, the abstractions are right associative,
 whereas the term application is left associative. The calculus examined by us is the simply typed one, the typing relation is presented in
 the next definition.

\begin{defi}[Type system]
\hfill
\begin{enumerate}
\item The types are
built from atomic formulas (or, in other words, atomic types) and the constant symbol $\bot$ with
the connector $\ra$. As usual for every type $A$, $\neg A$ is an abbreviation for $A \ra \bot$.
\item The length of a type $A$ (denoted by $lh(A)$) is defined as the number of arrows of $A$.
\item In the definition below, $\G$ denotes a
(possibly empty) context, that is, a finite set of declarations of the
form $x:A$ (resp. $\a :\neg A$) for a $\l$-variable $x$ (resp. a
$\mu$-variable $\a$) and type $A$ such that a $\l$-variable $x$
(resp. a $\m$-variable $\a$) occurs at most once in an expression
$x:A$ (resp. $\a : \neg A$) of $\G$. The typing rules are:
$$\F{}{\G, x:A\;\vdash\; x:A}\; ax$$

\vspace{.3cm}

\begin{tabular}{l l} $\;\;\;\;\;\;\;\;\;\; \F {\G, x:A\;\v\; M:B}{\G \;\v\; \l x.M:A\ra B}
\; {\ra}_{i}$ &
$\;\;\;\;\;\;\;\; \F {\G \;\v\; M:A\ra B\;\;\;\;\; \G \;\v\;
N:A}{\G
\;\v\; M\,N:B}\; {\ra}_{e}$
\end{tabular}

\vspace{.5cm}

\begin{tabular}{l l} $\;\;\;\;\;\;\;\;\;\; \F {\G, \a :\neg A \v\; M:A}{\G , \a :\neg A \;\v\; [\a] M : \bot}
\; {\bot}$ &
$\;\;\;\;\;\;\;\; \F {\G, \a :\neg A \;\v\; M: \bot
}{\G \v\;\m \a. M:A}\;{\m}$
\end{tabular}
\vspace{.5cm}

\item We will say that a term $M$ is typable with $A$, if there is a set of
declarations $\G$ such that $\G \;\v\; M:A$ holds.
\end{enumerate}
\end{defi}

\begin{defi}[Reduction rules]
\hfill
\begin{enumerate}
\item We have four kinds of redexes
\begin{itemize}
\item a $\b$-redex : term of the form $(\l x.M) N$,
\item a $\mu$-redex : term of the form $(\m \a. M) N$,
\item a $\r$-redex : term of the form $[\a]\m \b. M$,
\item a $\theta$-redex : term of the form $\mu \a. [\a]M$ and $\a \not \in Fv(M)$.
\end{itemize}
We denote by $NF$ the set of normal forms i.e. terms without redex.

\item The reduction rules are as follows:
\begin{itemize}
\item The $\b$-reduction rule is $(\l x.M) N \ra_{\b} M[x:=N]$

where $M[x:=N]$ is obtained from $M$ by replacing every
$x$ in $M$ by $N$.

\item  The $\m$-reduction rule is $(\m \a. M) N \ra_{\mu} \m \a. M[\a :{=}_{r}N]$

where $M[\a :{=}_{r}N]$ is obtained from $M$ by replacing every
subterm in $M$ of the form $[\a]U$ by $[\a](U\;N)$.

\item  The $\r$-reduction is $[\a]\m \b. M \ra_{\rho} M[\b :=\a ]$

where $M[\b :=\a ]$ is obtained by exchanging in $M$ every free
occurrence of $\b $ for $\a$.

\item  The $\theta$-reduction is $\mu \a. [\a]M \ra_{\theta} M$
provided $\a \not \in Fv(M)$.
\end{itemize}

\item Let $R$ be a redex of $M$. We write $M \ra^R N$ if $N$ is the term $M$
after the reduction of $R$. If $M = M_1\ra^{R_1}M_2\ra^{R_2}\dots \ra^{R_n}M_{n+1} = N$,
then $\si=[R_1,\dots ,R_n]$ denotes this reduction sequence, $n=|\si|$ and we write $M \fl^{\si} N$.

\item Let $\si$, $\nu$ be (possibly empty) sequences of reductions. Then
$\si \# \nu$ denotes their concatenation.
Let $\si =[R_1,\dots ,R_n]$. We denote by $\si [x:=M]$ (resp. $\si [\a :=_rM]$) the reduction sequence
$[R_1[x:=M],\dots ,R_n[x:=M]]$ (resp. $[R_1[\a :=_rM],\dots ,R_n[\a :=_rM]]$).
Moreover, let $\si [\a :=\b ]$ denote the reduction sequence
$[R_1[\a :=\b ],\dots ,R_n[\a :=\b ]]$.

\item As it is customary, by a reduction step we mean the closure of the
reduction relation compatible with respect to the term formation
rules. In general $\ra$
denotes the compatible closure of a reduction relation, or that of
the union of some set of relations, while by $\fl$ we mean the
reflexive, transitive closure of $\ra$. Sometimes we write $M \fl^n N$ if $M$
is reduced with $n$ steps of reductions to $N$.

\item  If $M$ is strongly normalizing i.e. $M$ has no infinite reduction sequences, then, by K\"onig's infinity lemma,
$\eta(M)$ will denote the length of the longest reduction sequence starting from $M$.
\end{enumerate}
\end{defi}

We present below some theoretical properties of the $\lmrt$-calculus.

\begin{thm}[Church-Rosser property]
Let $M_1$, $M_2$ and $M_3$ be terms such that $M_1\fl M_2$
and $M_1\fl M_3$. Then there exists an term $M_4$ for which $M_2\fl
M_4$ and $M_3\fl M_4$.
\end{thm}

A proof of the above assertion can be found in the papers of Parigot
\cite{Par1}, Py \cite{Py} or Rozi\`ere \cite{Roz}.
Py \cite{Py} expounded the question to a greater extent
together with the results belonging to the topic.

\begin{prop}[Type preservation property]
Let $M$, $N$ and $A$, $\G$ be such that $\G\;\v\; M:A$ and $M\fl
N$. Then $\G \;\v\; N:A$.
\end{prop}

The property can be verified by double induction on the length of
the reduction sequence $M\fl N$ and the complexity of $M$.

\begin{thm}[Strong normalization]\label{ch2:snforlm}
If $M$ is a typable term, then $M$ is strongly normalizing i.e. every reduction
sequence starting from $M$ is finite.
\end{thm}

There are several proofs of this result in the literature.
Consider, for example, Parigot \cite{Par3}, David and Nour
\cite{Dav-Nou1}. In \cite{de Gro2}, de Groote proves the strong normalization
of the simply typed $\l \m$-calculus extended with terms of
conjunctive and disjunctive types, respectively. He does not
consider the $\rho$- and $\th$-reduction rules in his calculus.\\
Albeit, we aim to find an upper bound for the reduction sequences of the $\l \m$-calculus enriched with the $\rho$-
and $\theta$-rules, as a by-product we also obtain a proof for the strong normalization of the $\l \m \rho \theta$-calculus.
We consider fewer simplification rules than Saurin \cite{Sau}, however, our notion of
standardness is formulated in a different form which enables us to prove statements concerning standard reduction sequences by
induction on the complexity of terms.

In this paper we consider only simply typed $\l\m$-terms. The typing relations involve
that a $\m$-variable cannot have but one argument, that is, we are not allowed to formulate terms of the form $(([\a] M)\;N)$,
where $\a$ is a $\mu$-variable and $M$, $N$ are arbitrary terms.

\subsection{Head and leftmost reductions}

In order to proceed to the standardization theorem, we define the notions of head- and leftmost reduction sequences.
Both are special cases of the standard reduction sequences discussed in the next section.

\begin{defi}
\hfill
\begin{enumerate}
\item Let $M$ be a term and $\ora{P}$  a possibly empty sequence of terms $P_1,\dots, P_n$.
We write $(M\; \ora{P})$ for the term $(\dots ((M \;P_1) \; P_2) \; \dots P_n)$, which also is denoted by $(M \;P_1\dots P_n)$.
\item Let $M=(P_1\dots P_n)=(P_1\;\ora{P})$, with a possibly
empty sequence of terms \smash{$\ora{P}$}.
Then, for $2 \leq i\leq n$, we write $P_i\in \ora{P}$ and we call $P_i$ $(2 \leq i\leq n)$ the components of $\ora{P}$
or the arguments of $P_1$.
\item Let $\ora{P} = P_1\ldots P_n$. We write $\ora{P}\fl^\si\ora{P'}$, where $\si=\si_1\#\ldots \#\si_n$ is
such that $P_i\fl^{\si_i}P_i'$ $(1 \leq i\leq n)$.
\end{enumerate}
\end{defi}

\begin{lem}\label{ch2:repr}
Every term $M$ of the simply typed $\lmrt$-calculus can be written
uniquely in one of the following forms.
\hfill
\begin{enumerate}
\item $M$ is a variable,
\item $M=\l x.M_1$, or $M=\m \a. M_1$ and $M$ is not a $\th$-redex, or $M=[\a]M_1$ and $M$ is not a $\rho$-redex,
\item $M=(x\;M_1\ora{P})$,
\item $M=(\l x.M_1)M_2\ora{P}$, or  $M=(\m \a.M_1)M_2\ora{P}$,
\item $M=[\a]\m \b. M_1$, or $M=\m \a. [\a]M_1$ and $\a \notin Fv(M_1)$.
\end{enumerate}
\end{lem}

\begin{proof}
By induction on $comp(M)$.
\end{proof}

In the following definitions the functions $hr$ and $lr$ are undefined in the cases not mentioned explicitly.

\begin{defi} \label{ch2:hdrdx}
\hfill
\begin{enumerate}
\item The head-redex of a term $M$, in notation $hr(M)$,
is defined as follows.
\begin{itemize}
\item $hr(\l x.M)=hr(M)$,
\item $hr(\m \a. M)=\m \a. M$ if $\m \a M$ is a $\th$-redex, and $hr(\m \a.
M)=hr(M)$ otherwise,
\item $hr([\a]M)=[\a]M$ if $[\a]M$ is a $\rho$-redex, and
$hr([\a]M)=hr(M)$ otherwise,
\item $hr((\m \a. M_1)M_2\ora{P})=(\m \a.
M_1)M_2$,
\item $hr((\l x.M_1)M_2\ora{P})=(\l x.M_1)M_2$.
\end{itemize}

If $M=(\m \a. M_1)M_2\ora{P}$, then a critical pair of redexes can emerge provided $(\m
\a. M_1)$ is a $\th$-redex as well. In this situation we always
choose the $\m$-redex $(\m \a. M_1)M_2$ as the head-redex of $M$.
\item Let $M_1\ra^{R_1}M_2\ra^{R_2}\dots \ra^{R_n}M_{n+1}$. Then $\si
=[R_1,\dots ,R_n]$ is a head-reduction sequence, if, for each
$1\leq i\leq n$, $R_i$ is the head-redex of $M_i$. We denote by $M\fl_{hd}N$ the fact that $M$
reduces to $N$ via a head-reduction sequence.
\end{enumerate}
\end{defi}

\begin{defi}\label{ch2:lmredex}
\hfill
\begin{enumerate}
\item The leftmost-redex of a term $M$, in notation $lr(M)$,
is defined as follows.
\begin{itemize}
\item $lr(\l x.M)=lr(M)$,
\item $lr(\m \a. M)=\m \a. M$ if $\m \a M$ is a $\th$-redex, and $lr(\m \a.
M)=lr(M)$ otherwise,
\item $lr([\a]M)=[\a]M$ if $[\a]M$ is a $\rho$-redex, and
$lr([\a]M)=lr(M)$ otherwise,
\item $lr((\m \a. M_1)M_2\ora{P})=(\m \a.
M_1)M_2$,
\item $lr((\l x.M_1)M_2\ora{P}))=(\l x.M_1)M_2$,
\item $lr(x\;M_1 \;M_2 \dots M_n)=lr(M_i)$ provided $M_i\notin NF$ and
$M_j\in NF$ $(1\leq j\leq i-1)$.
\end{itemize}
\item A reduction sequence $M_1\ra^{R_1}M_2\ra^{R_2}\dots
\ra^{R_n}M_{n+1}$ is the leftmost-reduction sequence
from $M_1$ to $M_{n+1}$ if $R_i$ is the leftmost-redex of $M_i$
$(1\leq i\leq n)$. We denote by $M\fl_{lrs}N$ the fact that $M$
reduces to $N$ via a leftmost-reduction sequence. Then the
reduction sequence itself is denoted by $lrs(M\fl N)$.
If $M\fl^{\si}N$ and $\si$ is a leftmost reduction sequence, then  $\si$ is unique.
\end{enumerate}
\end{defi}

Following the tradition of relating head reduction sequences to leftmost reduction sequences in the case of the $\l$-calculus \cite{Baren},
we compare briefly the two notions of reductions.

\begin{lem}\label{ch2:hdlm}
Every head-reduction sequence is a leftmost-reduction sequence.
\end{lem}

\begin{proof}
A straightforward induction on the complexity of the term, comparing the various subcases
of Definitions \ref{ch2:hdrdx} and \ref{ch2:lmredex}.
\end{proof}

We give a sketch of the proof that every leftmost-reduction sequence is the concatenation
of head-reduction sequences, however.
To this end, we first settle what we mean by a term being in head-normal form.

\begin{defi}
A term $M$ is in head-normal form (in notation $M\in HNF$), if one of the following cases is valid.
\begin{enumerate}
\item $M=\l x.M_1$ and $M_1\in HNF$,
\item $M=(x\;M_1\ldots M_k)$,
\item $M=\mu \a. M_1$, $M$ is not a $\theta$-redex and $M_1\in HNF$,
\item $M=[\a] M_1$, $M$ is not a $\rho$-redex and $M_1\in HNF$.
\end{enumerate}
We say that $M'\in HNF$ is a head-normal form of $M$, if $M \fl_{hd} M'$.
Observe that, since the typed $\lmrt$-calculus is strongly normalizing, every term has a unique head-normal form.
\end{defi}

Prior to detailing the connection between leftmost reduction and head reduction, we introduce a new notion.

\begin{defi}
Let $M\in HNF$.
\begin{enumerate}
\item The core of $M$, in notation $core(M)$, is defined as follows.
\begin{itemize}
\item If $M=\l x. M_1$, or $M=\m \a. M_1$, or $M=[\a]M_1$, then $core(M)=core(M_1)$.
\item If $M=x$, or $M=(M_1\;M_2)$, then $core(M)=M$.
\end{itemize}
Observe that, if $M\in HNF$, $core(M)$ can be obtained from $M$ if we omit the initial $\l$-, $\m$-prefixes
or $\m$-variables standing in front of $M$.
\item Assume $core(M)=(x\;\ora{P})$ with a possibly empty $\ora{P}$.
Then we call the components of $\ora{P}$ the components of $M$, as well.
\end{enumerate}
\end{defi}

Intuitively, a leftmost reduction sequence is a head reduction sequence until the term reaches a head normal form.
At this point, the leftmost reduction sequence is the concatenation of leftmost reduction sequences of the components.
This is the content of the lemma below.

\begin{lem}
Let $M\fl^{\si}M''$ be a leftmost reduction sequence and assume $M''\in NF$. Then there exists $M'\in HNF$ and $\si'$, $\si''$
such that $M\fl^{\si'}M'\fl^{\si''}M''$, where $\si'$ is a head reduction sequence and,
if $core(M')=(x\;M_1'\ldots M_k')$, then $\si''=\nu_1\#\ldots\#\nu_k$ such that
$core(M'')=(x\;M_1''\ldots M_k'')$, $M_i'\fl^{\nu_i}M_i''$ and $\nu_i$ are leftmost reduction sequences $(1\leq i\leq k)$.
\end{lem}

\begin{proof}
By induction on $comp(M)$ taking into account the subcases of Definition \ref{ch2:lmredex}.
Let $M\fl_{lrs}^{\si}M''$ and assume $\si=[R]\#\widehat{\si}$.
A straightforward observation of the points of Definitions \ref{ch2:hdrdx} and \ref{ch2:lmredex} gives that,
if $M''\notin HNF$, then $R$ is the head redex of $M$. Hence we may assume $M\fl^{\si''}M''$,
where $M\in HNF$. Then, by induction on $comp(M)$, we obtain that we may suppose that $M=x$ or $M=(x\;M_1\ldots M_n)$.
By Definition \ref{ch2:lmredex}, both assumptions immediately yield the result.
\end{proof}

\subsection{Other definitions}

We define in this subsection the notion of a $\l\m$I-term and a $\l\m$I-redex, together with some main properties of the $\l\m$I-calculus.

\begin{defi}
\hfill
\begin{enumerate}
\item The set of $\lm I$-terms is defined inductively as follows:
\begin{itemize}
\item $x$ is a $\lm I$-term,
\item $\l x.M$ is a $\lm I$-term provided $M$ is a $\lm I$-term and $x\in Fv(M)$,
\item $(M\;N)$ is a $\lm I$-term if $M$ and $N$ are $\lm I$-terms,
\item $\mu \a. M$ is a $\lm I$-term provided $M$ is a $\lm I$-term and $\a \in Fv(M)$,
\item $[\a]M$ is a $\lm I$-term provided $M$ is a $\lm I$-term.
\end{itemize}
\item If $M=(\l x.M_1)M_2$ (resp.  $M=(\mu \a. M_1)M_2$, where $\l x. M_1$ (resp. $\mu \a. M_1$) and $M_2$ are $\l\m$I-terms,
then $M$ is called a $\l\m$I-redex.
\end{enumerate}
\end{defi}
It is easy to see that, if $M$ is a $\l \m$I-term and $M \fl M'$, then $M'$ is a $\l \m$I-term. Thus,
it is clear that this calculus also has the following three properties:
Church-Rosser-property, type preservation and strong normalization.\\

The next section is concerned with a standardization result in the $\l\m\rho\th$-calculus.
In the sequel, we are going to use the notions of subterms, redexes, reduction sequences,
residuals etc. in an intuitive manner.

A reduction sequence $M_1\ra^{R_1}M_2\ra^{R_2}\ldots\ra^{R_n} M_{n+1}$ is a sequence of terms and redex occurrences,
where $M_{i+1}$ is obtained by reducing with $R_i$ in $M_i$ $(1\leq i\leq n)$. In what follows, by an abuse of notation,
a reduction sequence will be referred to without noting explicitly the exact occurrences of the redexes in the terms,
if they are clear from the context. We give a short account of the intuitive notions for residuals and involvement
of redexes.

\begin{defi}
\hfill
\begin{enumerate}
\item Let $M\ra^RM'$ be a reduction step.
\begin{enumerate}
\item If $R=(\l x.R_1)R_2 < M$, then $R$ and $\l x.R_1$ have no residuals, otherwise,
if $(\l x.R_1)R_2 < U\leq M$, we obtain the residual of $U$ by exchanging $(\l x.R_1)R_2$ for $R_1[x:=R_2]$ in $U$.
When $U\leq R_1$, then we obtain the residual by substituting each occurrence of $x$ by $R_2$ in $U$. In the case of $U\leq R_2$,
the residual of $U$ is the same, only its position changes in $M'$: its index will be one of the indices
of a former occurrence of $x$ in $R_1$. If $R$ and $U$ are disjoint, then the residual of $U$ is $U$ itself.

\item The situations are analogous in the cases of the other redexes: if $R=(\m\a. R_1)R_2$, then $R$ has no residual,
if $R=[\a]\m\b. R_1$, then $R$ and $\m\b. R_1$ have no residuals, and, finally, if $R=\m\a.[\a]R_1$,
then $R$ and $[\a]R_1$ have no residuals. Besides these afore-said cases, if $R< U\leq M$,
then we obtain the residual of $U$, if we execute the redex $R$ in $U$. When $U< R$ and $U$ has a residual:
if $R=(\m\a. R_1)R_2$ and $U\leq R_1$, we obtain the residual by recursively exchanging every subterm $[\a]P$
of $U$ by $[\a](P\;R_2)$. If $U\leq R_2$, then the residual is $U$ only its index changes in
$M'$. Otherwise, for $U\leq R_1$ and $R=[\a]\m\b. R_1$, the residual is $U[\b:=\a]$. If $U\leq R_1$ and $R=\m\a.[\a]R_1$,
then the residual is $U$ just the index is modified in $M'$.
\end{enumerate}
\item Residuals of terms with respect to reduction sequences are defined in a recursive way: we obtain the residuals with respect
to a reduction sequence if we compute the residuals of the residuals with respect to subsequences of the reduction sequence.
\item Let $\si$ be the reduction sequence $M_1\ra^{R_1}M_2\ra^{R_2}\ldots\ra^{R_n} M_{n+1}$, assume $R\leq M_1$ is a redex.
Then we say that $R$ is involved in $\si$, if there is an $1\leq i\leq n$ such that $R=R_i$ and $R_i$ is a residual of $R$
with respect to $M_1\ra^{R_1}\ldots\ra^{R_{i-1}}M_{i}$.
\end{enumerate}
\end{defi}

In what follows, most of the proofs will follow an induction on lexicographically ordered tuples of integers.
Ordering of tuples is understood in the usual lexicographic manner: $(n,m)\leq (n',m')$ iff either $n<n'$
or $n=n'$ and $m\leq m'$.

\section{Standardization for the \texorpdfstring{$\l\m\rho\th$}{lambda-mu-rho-theta}-calculus}

Our results concerning the standardization of the $\lm$-calculus are not the strongest ones. In fact, some of our statements
are valid for the $\lm$I-calculus only. A standardization result can be found in the paper of Saurin \cite{Sau}, where, besides
the rules mentioned in our article, some other rules are taken into account. Our only concern with the standardization is our aim
to find an upper bound for the reduction sequences of the $\lm$-calculus.
In the present subsection we define a notion of a standard reduction sequence for the $\lmrt$-calculus and find some assertions
concerning their lengths.
Many of the proofs are adaptations of the ones related to the simply typed $\l$-calculus in
\cite{Xi}. The result itself, however, is not a simple generalization of Xi's method.
In the presence of $\m$-, $\rho$- and $\theta$-reductions overlapping redexes mean the greatest obstacle to a straightforward formulation of standardness.
We suggest the following solution to this problem. We define the notion of a standard reduction sequence such that every standard reduction sequence should
obey
the following properties: when a redex, which is simultaneously a $\m$- and a $\theta$-redex, is involved in a standard reduction sequence,
 we stipulate that the redex should be understood only as a $\m$-redex. Likewise, when a $\theta$-redex
would destroy a containing $\rho$-redex we prohibit reducing the $\theta$-redex, and when a $\rho$-redex would make
a containing $\th$-redex disappear, we forbid the $\rho$-redex until the $\theta$-redex exists.
These raise additional issues in the estimation of the lengths of standard reduction sequences:
we must take into account the numbers of arguments of such $\th$-redexes that are simultaneously $\m$-redexes and we must exclude some reduction sequences
from the set of standard reduction sequences in order to deal with the overlapping $\rho$- and $\theta$-redexes.
These considerations are reflected in the definition of a standard reduction sequence and in the measure for a term presented
in Definition \ref{ch2:mr}.
We show that our suggestion for a solution
is appropriate: we can majorize every reduction sequence
by a standard reduction sequence of Definition \ref{ch2:stlmrt}.

We should remark that the widely known and intuitive definition requires of a
standard reduction sequence that no redex is a residual of a redex which lies to the left
of some other redex in the sequence \cite{Baren}. Instead of this, we
use a definition of a standard reduction sequence similar to the one applied in \cite{Dav},
which enables us to prove properties concerning standard reduction sequences by induction
on the complexity of terms.

\subsection{Standard reduction sequences in the \texorpdfstring{$\lmrt$}{lambda-mu-rho-theta}-calculus}

In this subsection we define the notion of a standard $\bmrt$-reduction sequence and present some elementary
lemmas concerning properties of standard reduction sequences. In the definition below, we clarify what we mean by a standard $\bmrt$-reduction sequence.
The definition is structured by induction on the lexicographically ordered pair $(|\si|,comp(M))$.
\begin{defi} \label{ch2:stlmrt}
A reduction sequence $M\fl^{\si}N$ is standard if, either it is empty, or one of the following cases holds.
\begin{enumerate}
\item $M=\l x.M_1$, $N=\l x.N_1$, $M_1\fl^{\si}N_1$ and $\si$ is standard.
\item If $M=\m\a. M_1$, let $M\fl^{\si}N$ be $M=P_1\ra P_2\ra\ldots\ra P_{k+1}=N$.
\begin{enumerate}
\item Either $N=\m\a. N_1$, $M_1\fl^{\si}N_1$ and $\si$ is standard and none of $P_j$ is a $\th$-redex $(1\leq j\leq k+1)$,
\item or let $M\fl^{\si'}\m\a.[\a]M'=P_j$ such that $P_j$ is the first term in the sequence which is a $\th$-redex and
$\si'$ is standard and
\begin{enumerate}
\item either $\m\a.[\a]M'=P_j\ra_{\th}M'\fl^{\si''}N$,
\item or $\m\a.[\a]M'=P_j\fl^{\si'''}\m\a.[\a]N'$ such that $M'\fl^{\si'''}N'$,
\end{enumerate}
where $\si''$ and $\si'''$ are standard.
\end{enumerate}
\item If $M=[\a]M_1$, let $M\fl^{\si}N$ be $M=P_1\ra P_2\ra\ldots\ra P_{k+1}=N$.
\begin{enumerate}
\item Either $N=[\a]N_1$, $M_1\fl^{\si}N_1$ and $\si$ is standard and none of $P_j$ is a $\rho$-redex $(1\leq j\leq k+1)$,
\item or let $M\fl^{\si'}[\a]\m\b. M'=P_j$ such that $P_j$ is the first term in the sequence which is a $\rho$-redex
and $\si'$ is standard and
\begin{enumerate}
\item either $[\a]\m\b. M'=P_j\ra_{\rho}M'[\b:=\a]\fl^{\si''}N$,
\item or $[\a]\m\b. M'=P_j\fl^{\si'''}[\a]\m\b. N'$
such that $M'\fl^{\si'''}N'$,
\end{enumerate}
where $\si''$ and $\si'''$ are standard.
\end{enumerate}
\item $M=(\l x.M_1) M_2 \ldots M_n$ and
\begin{enumerate}
\item either $M\ra_{\b}(M_1[x:=M_2]\dots M_n)\fl^{\si_1}N$ and $\si_1$ is standard,
\item or $M\fl^{\si_1}(\l x.N_1)M_2\dots M_n\fl^{\si_2}(\l x.N_1) N_2\dots M_n\fl^{\si_3}\ldots \fl^{\si_n}$\\
$(\l x.N_1) N_2\ldots N_n=N$ and $\si_i$ $(1\leq i\leq n)$ are standard.
\end{enumerate}
\item $M=(\m \a. M_1) M_2\ldots M_n$ and
\begin{enumerate}
\item either $M\ra_{\m}(\m\a. M_1[\a:=_rM_2]\ldots M_n)\fl^{\si_1}N$ and $\si_1$ is standard,
\item or $M\fl^{\si_1}(\m \a. N_1)M_2\ldots M_n\fl^{\si_2}(\m\a. N_1)N_2\ldots M_n \fl^{\si_3}\ldots \fl^{\si_n}$\\
$(\m \a. N_1) N_2\ldots N_n =N$ and $\si_i$ $(1\leq i\leq n)$ are standard.
\end{enumerate}
\item $M=(x\;M_1\ldots M_n)$,
$M=(x\;M_1\ldots M_n)\fl^{\si_1}(x\;N_1\ldots M_n)\fl^{\si_2}\ldots\fl^{\si_n}$\\
$(x\;N_1\;N_2\ldots N_n)=N$
and $\si_i$ $(1\leq i\leq n)$ are standard.
\end{enumerate}

In the rest of this paper, we may treat a reduction sequence $\si$ as a list
of the terms in $\si$ or sometimes as the list of the redex occurrences of the
reduction sequence. In accordance with this, given a standard reduction sequence
$M_1\ra^{R_1}M_2\ra^{R_2}\ldots\ra^{R_n} M_{n+1}$, we may say that the sequence $M_1,\ldots ,M_{n+1}$
is standard (the redex occurrences are implicitly understood in $M_i$), or we may talk about the same
thing by just saying that the sequence $\si=[R_1,\ldots,R_n]$ is standard. In notation: $\si\in St$.
\end{defi}

We illustrate some of the difficulties in the example below, when we want to assert statements about standard reduction sequences.

\begin{exa}
Let $M=(\m\a.[\a]\l u.(\m \b.[\b]\l y.x)[\a]x)x$. Then, if we choose the $\theta$-redex $\m \b.[\b]\l y.x$,
we obtain $M \ra_{\th}(\m\a.[\a]\l u.(\l y.x)[\a]x)x$, and,
since we are not allowed to reduce the redex $(\l y.x)[\a]x$, there are no more reductions provided we
restrict ourselves to standard ones. On the other hand, if we choose the $\m$-redex $(\m\b.[\b]\l y.x)[\a]x$,
then $M \ra_{\m} (\m\a.[\a]\l u.\m \b.[\b](\l y.x)[\a]x)x \ra_{\b} (\m\a.[\a]\l u.\m \b.[\b]x)x$\\
$\ra_{\mu}\m\a.[\a](\l u.\m \b.[\b]x)x\ra_{\theta} (\l u.\m \b.[\b]x)x \ra_{\b}\m \b[\b]x \ra_{\th}x$ is standard.
\end{exa}

The above definition prevents standard reduction sequences from having overlapping $\th$- and $\rho$-redexes that could eliminate each other.
Moreover, our definition of standardness is such that it gives rise to the following distinction between standard reduction sequences,
at least in the case of the $\lm$I-calculus:
given a $\lm$I-term, if the head redex exists, then a standard reduction sequence either begins with the head redex or the head redex
has a unique residual in the resulting term, which is the head redex of the result itself. This will be demonstrated in Lemmas \ref{ch2:hdrdxlem} and \ref{ch2:invhdrdx}.
On the other hand, in the general case, the situation is a little more complicated as Examples \ref{ch2:exlmi1} and \ref{ch2:exlmi2} show. Example \ref{ch2:exlmi1}
even demonstrates that, in the general case, in the presence of the $\th$-rule, a standard reduction sequence is not necessarily left to right, in contrast with
the case of the $\l$-calculus.

Our aim in this section is to obtain a standardization theorem for the $\lmrt$-calculus,
together with an upper bound on the lengths of the standard reduction sequences.
To this end, we state and prove some auxiliary propositions first concerning the behaviour
of standard reduction sequences and then we present some lemmas providing upper bounds for the lengths
of reduction sequences starting from terms obtained as the results of substitutions.

We state our first theorem saying that left-most reduction sequences are special cases of standard reduction sequences.

\begin{thm}
Every leftmost reduction sequence is standard.
\end{thm}
\begin{proof}
Immediate from Definitions \ref{ch2:lmredex} and \ref{ch2:stlmrt}.
\end{proof}

The following lemma states that a reduction sequence, which consists of a head reduction sequence followed
by a standard reduction sequence is itself standard.

\begin{lem} \label{ch2:sthdrdx}
Let $M\fl^{\si'}M'\fl^{\si''}M''$ such that $\si'$ is a
head-reduction sequence and $\si''$ is standard. Then $\si =\si'\#
\si''$ is standard.
\end{lem}

\begin{proof}
Let $\si'=[R]\# \nu$. We prove the result by induction on $(|\si'| , comp(M))$, taking into account the various points of
Definition \ref{ch2:hdrdx}.
Assume $|\nu|=0$. We deal with two of the cases only.
\begin{enumerate}
\item $M=[\a]M_1$.
\begin{enumerate}
\item Assume $M$ is a $\rho$-redex, then $M_1 = \mu \b. M_2$ and $M\ra_{\rho}M_2[\b:=\a]\fl^{\si''}M''$
is standard by point 3 of Definition \ref{ch2:stlmrt}.
\item If $R\leq M_1$, then the induction hypothesis applies.
\end{enumerate}
\item $M=(\m \a. M_1)M_2\dots M_n$. In this case the head redex of $M$ is $(\m\a. M_1)M_2$.
Thus $M\ra_{\m}(\m\a. M_1[\a:=_rM_2]\dots M_n)\fl^{\si''}M''$ and $\si$ is standard by point 5 of Definition
\ref{ch2:stlmrt}.
\end{enumerate}
The cases when $|\nu|>0$ follow from the induction hypothesis.
\end{proof}

The lemma below is a technical one, it will be useful for verifying that, if we are given terms $M$, $M'$ and $N$, $N'$ such that
$M\fl_{st}M'$ and $N\fl_{st}N'$, then it is also true that $M[x:=N]\fl_{st}M'[x:=N']$ and
$M[\a:=_rN]\fl_{st}M'[\a:=_rN']$, respectively.

\newpage
\begin{lem} \label{ch2:sthdrdx2}\leavevmode

\begin{enumerate}
\item Let $M=(M_1\;M_2)\fl^{\si}(\m\a. M_3)M_2\ra_{\mu}(\m\a. M_3[\a:=M_2])\fl^{\nu}N$
such that $\si$, $\nu\in St$ and suppose $\m\a. M_3$ is the first term of the reduction sequence $M_1\fl^{\si} \m\a. M_3$
of the form $\m \b. P$.
Then $\xi=\si\#[(\m\a. M_3)M_2]\#\nu$ is standard.
\item Let $M=(M_1\;M_2)\fl^{\si}(\l x.M_3)M_2\ra_{\mu}M_3[x:=M_2]\fl^{\nu}N$ such that $\si$, $\nu\in St$
and suppose $\l x.M_3$ is the first term of the reduction sequence  $M_1 \fl^{\si} \l x.M_3$
of the form $\l y.P$. Then $\xi=\si\#[(\l x.M_3)M_2]\#\nu$ is standard.
\end{enumerate}
\end{lem}

\begin{proof}
We deal only with case 1. We examine some of the interesting cases.
The proof goes by induction on $(|\si|,comp(M))$ taking into account the various points of Definition \ref{ch2:stlmrt}.
\begin{enumerate}
\item If $\si$ is standard by virtue of point 2 of Definition \ref{ch2:stlmrt}, that is, $M=\m\a. M_1$,
the only possibility is when $R= \m\a.[\a]R_1$ is the first $\th$-redex in the sequence.
But then $R$ is in fact $M$, and $[R]\#\nu$ is standard by definition.
\item Let $\si$ be standard by reason of point 5 of Definition \ref{ch2:stlmrt}.
By assumption, the only possibility is $M=(\m \a. P_1)P_2\ldots P_n$ and $\si=[(\m \a. P_1)P_2]\#\si'$.
The induction hypothesis can be applied to $\si'$ and $M'=(\m \a. P_1[\a:=_rP_2]\ldots P_n)$.\qedhere
\end{enumerate}
\end{proof}

The following lemma gives us some information on the form of a term which is a reduct obtained
by a standard reduction sequence.

\begin{lem}\label{ch2:dcm}
Let $M \fl^{\si}M'$ such that $\si$ is standard. Assume the
head-redex $hd(M)$ of $M$, if it exists, is not involved in $\si$.
Then the following statements are true.
\begin{enumerate}
\item If $M=\l x.M_1$, then $M'=\l x.M_1'$, where $M_1\fl^{\si}M_1'$.
\item If $M=[\a]M_1$, then $M'=[\a]M_1'$, where $M_1\fl^{\si}M_1'$.
\item
If $M=(M_1 \dots M_n)$, then there are standard $\si_1,\dots ,\si_n$ and terms $M_1',\dots ,M_n'$ such
that $M_i\fl^{\si_i}M_i'$ $(1\leq i\leq n)$, $M'=(M_1'\dots M_n')$ and $\si =\si_1 \# \dots \# \si_n$.
\item If $M = \m\a. M_1$ is a $\l\m$I-term, then $M'= \m\a. M_1'$, where $M_1\fl^{\si}M_1'$.
\end{enumerate}
\end{lem}

\begin{proof}
By induction on $(|\si| , comp(M))$. We consider some of the typical cases.
\begin{enumerate}
\item $M=(\m \a. M_1)M_2\ldots M_n$. Since the head redex $(\m\a. M_1)M_2$ is not involved in $\si$,
point 5 of Definition \ref{ch2:stlmrt} yields the result.
\item $M=\m\a. M_1$. If $M=\m\a[\a]M_2$ is a $\th$-redex,
then, since $M$ is not involved in $\si$, point 2 of Definition \ref{ch2:stlmrt} yields that $M_2\fl^{\si}N_2$
such that $N=\m\a.[\a]N_2$. Assume now $M$ is not a $\th$-redex.
Thus $M$ is either not of the form $\m\a.[\a]M_2$ such that $\a\notin Fv(M_2)$ or $M=\m\a.[\a]M_2$ and $\a\in Fv(M_2)$.
Then $hd(M)=hd(M_1)$ and, applying the induction hypothesis to $M_1$, it is straightforward to check
that either $M$ cannot reduce to a term of the form $\m\a.[\a]M''$, or $M\fl \m\a.[\a]M''$ and $\a\in Fv(M'')$.
Again, by point 2 of Definition \ref{ch2:stlmrt} we obtain the result.\qedhere
\end{enumerate}
\end{proof}

We remark that the assumption of $M$ being a $\lm$I-term is crucial in Case 4 of
Lemma \ref{ch2:dcm} as the following example shows.

\begin{exa}\label{ch2:exlmi}
If $M=\m\a.[\a](\l x.x)((\l y.x)[\a]x)$, then $hd(M)=(\l x.x)((\l y.x)\;[\a]x)$.
Consider the standard reduction sequence
$\si$: $M\ra_\b \m\a.[\a](\l x.x)x \ra_\theta (\l x.x)x=M'$.
Then $hd(M)$ is not involved in $\si$, on the other hand, $M'$ is not of the form $\m\a. M_1'$.
\end{exa}

In the next two lemmas our common assumption is that $M$ is a $\l\m$I-term. The lemmas will serve as auxiliary
statements when we prove that a standard normalizing reduction sequence is unique in the case of $\l\m$I-terms.

\begin{lem} \label{ch2:hdrdxlem}
Let $M$ be a $\l\m$I-term. If $M\fl^{\si}M'$ is standard such that the head-redex $hd(M)$ of $M$ exists and
is not involved in $\si$, then the head-redex $hd(M')$ of $M'$ exists and it is the unique residual of $hd(M)$
with respect to $\si$.
\end{lem}

\begin{proof}
By induction on $(|\si|,comp(M))$, taking into account the various cases of
Definition \ref{ch2:hdrdx}. Let $\si=[R]\#\si'$. We assume $|\si'|=0$. We examine some of the cases.
\begin{enumerate}
\item $M=\m\a. M_1$.
\begin{enumerate}
\item $M=\m\a.[\a]M_2$.

Assume $M$ is a $\th$-redex. By assumption, $M_2\fl^{R}M_2'$ such that $M'=\m\a.[\a]M_2'$.
Then our assertion follows. We have also made use of point 2 (b) of Definition \ref{ch2:stlmrt}.

Assume now $M$ is not a $\th$-redex, which implies $\a\in Fv(M_2)$. Then $hd(M_1)=hd(M)$ is not $R$,
we have $M_2\fl^{R}M_2'$. Since $M$ is a $\l\m$I-term, $\a\in Fv(M_2')$ holds. Thus, if $M_1'=[\a]M_2'$, $hd(M_1')=hd(M')$,
by which, and the induction hypothesis, we have the result.

\item $M\neq \m\a.[\a]M_2$. Let $M_1\ra^{R}M_1'$, where $M'=\m\a. M_1'$.
By the induction hypothesis, $hd(M_1')$ exists and it is the unique residual of $hd(M_1)$, which is $hd(M)$.
We prove that $hd(M')=hd(M_1')$, by which our assertion follows. By Definition \ref{ch2:hdrdx},
it is enough to verify that $M'$ is not a $\th$-redex. Lemma \ref{ch2:dcm} shows that the only possibility is
$M=\m\a.[\a]M_2$ for some $M_2$ provided $M'$ is a $\th$-redex, but this was excluded by the assumption.
\end{enumerate}
\item $M=[\a]M_1$.
\begin{enumerate}
\item $M=[\a]\m\b. M_2$. Then our assumption and point 3 of Definition \ref{ch2:stlmrt} yields the statement.
\item $M$ is not a $\rho$-redex. Then Lemma \ref{ch2:dcm} ensures that $M'$ is not a $\rho$-redex either.
If $M_1\fl^{R}M_1'$, where $M'=[\a]M_1'$, then $hd(M)=hd(M_1)$ and $hd(M')=hd(M_1')$, by which,
together with the induction hypothesis,  our claim follows.
\end{enumerate}
\item $M=(\m\a. M_1)M_2\ldots M_n$. Since $hd(M)$ is not $R$, the only possibility is $M_i\ra^RM_i'$ for some $1\leq i\leq n$.
Hence our assertion follows.
\end{enumerate}
The case $|\si'|>0$ follows by the induction hypothesis.
\end{proof}

The assumption that $M$ is a $\lm$I-term is necessary in the above lemma, too.

\begin{exa}\label{ch2:exlmi1}
Let $M=\m\a.[\a](\l x.x)((\l y.x)[\a]x)$, as in Example \ref{ch2:exlmi}, then $hd(M)=(\l x.x)((\l y.x)[\a]x)$.
Consider the reduction sequence
$M \ra_\b \m\a.[\a](\l x.x)x=M'$.
In this case $hd(M')=M'$ and it is not a residual of $hd(M)$.
\end{exa}

\begin{lem} \label{ch2:invhdrdx}
Let $M$ be a $\l\m$I-term. If $M \fl^{\si} M'$ is standard and the head-redex $hd(M)$ of $M$ is involved in $\si$,
then $\si=[hd(M)]\# \si'$ for some $\si'$.
\end{lem}

\begin{proof}
The proof goes by induction on $|\si|$, considering the cases of Definition \ref{ch2:stlmrt}.
If $|\si|=1$, then the statement is trivial.
Assume $\si=[R]\# \si'$, where $|\si'|>0$, and let $hd(M)$, the head redex of $M$, be different from $R$.
Let $M\ra^{R}M''\fl^{\si'}M'$. By Lemma \ref{ch2:hdrdxlem}, the head redex $hd(M'')$ of $M''$ exists and it is the unique residual
of $hd(M)$ with respect to $R$. Then $hd(M'')$ is involved in $\si'$, thus, by the induction hypothesis,
we have $\si'=[hd(M'')]\# \si''$.
Now, by examining the various forms of $M$ according to Definition \ref{ch2:hdrdx},
 we can easily check that the above situation is impossible.
\end{proof}

Again, $M\in \lm$I is necessary for the statement of the previous lemma.

\begin{exa}\label{ch2:exlmi2}
Let $M=\m\a.[\a](\l x.x)((\l y.x)[\a]x)$, as in Example \ref{ch2:exlmi}, then $hd(M)=(\l x.x)((\l y.x)[\a]x)$.
Consider the standard reduction sequence $\si$,\\
$M \ra_\b \m\a.[\a](\l x.x)x \ra_\theta (\l x.x)x \ra_\b x=M'$.
Then $hd(M)$ is involved in $\si$ and, on the other hand, $\si$ is not of the form $[hd(M)]\# \si'$ for some $\si'$.
\end{exa}

\subsection{Calculating the bounds for substitutions}

In the following lemmas we examine how standardization is related to substitutions in relation to  $\l$- and $\mu$-variables.
In addition, we give estimations for the lengths of standard reduction sequences starting from terms given in the form of substitutions.
The lemmas in this subsection are indispensable for proving Lemma \ref{ch2:sta}, which is the standardization lemma.

The next lemma shows that the length of a standard reduction sequence is not modified by a $\l$-substitution, i.e.
we can find a standard reduction sequence of the same length for the substitutions.

\begin{lem} \label{ch2:sbtrx}
Let $M\fl^{\si}M'$ be standard, then there exists a $\nu\in St$ such that\\
$M[x:=N]\fl^{\nu}M'[x:=N]$ and $|\nu |=|\si |$.
\end{lem}

\begin{proof}
The proof goes by a straightforward induction on $(|\si |, comp(M))$ distinguishing the cases of Definition \ref{ch2:stlmrt}.
We deal only with the case $M=[\a]M_1$. We prove that the choice $\nu=\si[x:=N]$ is appropriate.
\begin{enumerate}
\item If $M\fl^{\si}M'$ with $M'=[\a]M_1'$ and $M_1\fl^{\si}M_1'$, then the induction hypothesis applies.
\item
\begin{enumerate}
\item If $M\fl^{\si_1}[\a]\m\b. P \ra_\rho P[\b:=\a]\fl^{\si_2}M'$
such that $[\a]\m\b. P$ is the first $\rho$-redex in the sequence,
then the induction hypothesis implies that
$M[x:=N]\fl^{\si_1[x:=N]}[\a]\m\b. P[x:=N]$ and $P[\b:=\a][x:=N]\fl^{\si_2[x:=N]}M'[x:=N]$ are standard,
moreover, $[\a]\m\b. P[x:=N]$ is the first $\rho$-redex in the sequence.
 We obtain the result immediately from Definition \ref{ch2:stlmrt}.
\item If $M\fl^{\si_1}[\a]\m\b. P \fl^{\si_2} [\a]\m\b. Q=M'$,
where $[\a]\m\b. P$ is the first $\rho$-redex in the sequence
then, by the induction hypothesis, $M[x:=N]\fl^{\si_1[x:=N]}[\a]\m\b. P[x:=N]$ and
$[\a]\m\b. P[x:=N]\fl^{\si_2}[\a]\m\b. Q[x:=N]=M'[x:=N]$ are standard, which,
considering Definition \ref{ch2:stlmrt}, yields the result.\qedhere
\end{enumerate}
\end{enumerate}
\end{proof}

In the sequel, we make preparations for the estimation of the upper bound of the length of a standard reduction sequence.
To this aim, we introduce quantitative notions in relation to reduction sequences.

\begin{defi} \label{ch2:occurnm} \label{ch2:sigmalph}
\hfill
\begin{enumerate}
\item Let $M$ be a term and $x$ (resp. $\a$) be a $\l$-variable (resp. $\mu$-variable).
Denote by $|M|_x$ (resp. $|M|_{\a}$) the number of occurrences of $x$ (resp. $\a$) in $M$.
\item Let $\si$ be the reduction sequence $M \ra^{R_1}M_1\ra^{R_2}\dots
\ra^{R_n}M_{n}$ and $\a \in Fv(M)$. Let $\lan \si
\ran _{(\rho , \a )}$ denote the number of $\rho$-reductions of the form $(\a\;\m \b.
P)$ in $\si$. Furthermore, let $\lan\si\ran_{\rho}$ be the number of $\rho$-redexes in $\si$.

\item If $M\fl^{\si}M'$, let us denote by
$\lan \si \ran_{\th}$ the number of $\th$-redexes in $\si$.
\item If $x \in Fv(M)$, let us denote by $sumarg(M,x)$ the sum of the number of arguments of each occurrence of $x$ in $M$.
It is easy to see that $sumarg(M,x)\leq comp(M)-1$.
\end{enumerate}
\end{defi}

Regarding the $\mu$-substitutions, the length of a standard reduction sequence can increase.
This is in connection with the standardization of reduction sequences initially containing $\rho$-redexes.

\begin{lem}  \label{ch2:sbtrm}
If $M\fl^{\si}M'$ is standard and $N_1,\dots ,N_k$ are terms for which $\a \notin Fv(N_i)$
$(1\leq i\leq k)$, then there exists a standard reduction sequence $\nu $ such that\\
$M[\a :=_rN_1]\dots [\a :=_rN_k]\fl^{\nu}M'[\a :=_rN_1]\dots [\a :=_rN_k]$ and
$|\nu |=|\si |+k\cdot \lan \si \ran_{(\rho,\a)}$.
\end{lem}

\begin{proof}
By induction on $( |\si| , comp(M))$.
The case $|\si|=0$ is trivial.
If $\si =  [R] \# \si'$, where $M\ra^{R}M''\fl^{\si'}M'$,
the only interesting case is $M=[\g]\mu \b. M_1\ra_{\rho}M_1[\b
:=\g ]=M''\fl^{\si'}M'$. If $\a \neq \g$, then the result follows from the induction hypothesis.
Otherwise, we have the following reduction sequence denoted by (1) :
$M[\a :=_rN_1]\dots [\a :=_rN_k] =
[\a](\mu \b. M_1[\a :=_rN_1]\dots [\a :=_rN_k])N_1\dots
N_k {} \fl_{\m}^k [\a]\m \b. M_1[\a:=_rN_1]\dots [\a :=_rN_k][\b
:=_rN_1]\dots [\b :=_rN_k]  \ra_{\rho} M_1[\b :=\a ][\a :=_rN_1]\dots [\a :=_rN_k]$,
from which the estimation for the length of $\nu$ follows. Assume $\si$ is standard, we prove by induction
on $( |\si|,comp(M))$ that $\nu$ is standard. We examine the cases of Definition \ref{ch2:stlmrt}.
We consider only the case when $\si$ is standard by reason of point 3 of Definition \ref{ch2:stlmrt}.
Let $M=[\g]M_1$. If $M_1\fl^{\si}M_1'$, the induction hypothesis applies. Otherwise,
there are standard $\si'$, $\si''$ such that $M\fl^{\si'}[\g]\m\b. M'' \fl^{\si''}M'$
and $[\g]\m\b. M''$ is the first term in the sequence which is a $\rho$-redex and either
$\si''=[[\g]\m\b. M'']\#\si'''$ for some $\si'''\in St$ or $M'=[\g]\m\b. M'''$ and $M''\fl^{\si''}M'''$.
Assume $\g=\a$. Then, by the induction hypothesis,
$M[\a :=_rN_1]\dots [\a :=_rN_k]\fl^{\nu'}[\a](\m\b. M''[\a :=_rN_1]\dots [\a :=_rN_k])N_1\ldots N_k$
is standard and $(\m\b. M''[\a :=_rN_1]\dots [\a :=_rN_k])N_1\ldots N_k$ is the first $\m$-redex in the sequence.
Then we apply the head reduction sequence of (1),
hence Lemma \ref{ch2:sthdrdx} involves that $\nu$ is standard. If $\g\neq \a$,
then $M[\a :=_rN_1]\dots [\a :=_rN_k]=[\g]M_1[\a :=_rN_1]\dots [\a :=_rN_k] \fl [\g]\m\b. M''[\a :=_rN_1]\dots [\a :=_rN_k]$.
If $\si''=[[\g]\m\b. M'']\#\si'''$, then Lemma \ref{ch2:sthdrdx2} applies.
Otherwise we obtain the result by the induction hypothesis.
\end{proof}

The situation in the lemma below is more complicated when we assume that we are provided a term together with a
standard reduction sequence emanating from that and we substitute the term in place of a variable of another term.
This is in relation with the possibility of creating new redexes. As we have seen earlier, sometimes we only obtain
an estimation for the lengths of the new standard reduction sequences.

\begin{lem} \label{ch2:cmp}
\hfill
\begin{enumerate}
\item If $N\fl^{\si}N'$ is standard, then there exists a standard reduction $\nu$ such that\\
$M[x:=N]\fl^{\nu}M[x:=N']$ and $|\nu |\leq |\si |\cdot |M|_x + sumarg(M,x)\cdot (\lan\si\ran_{\th}+\lan\si\ran_{\rho})$.
\item If $N\fl^{\si}N'$ is standard, then there exists a standard reduction $\nu $ such that\\
$M[\a :=_rN]\fl^{\nu}M[\a :=_rN']$ and $|\nu |=|\si |\cdot |M|_{\a}$.
\end{enumerate}
\end{lem}

\begin{proof}
Let us only deal with case 1.
The proof goes by a straightforward induction on $comp(M)$. We lean on the points of Definition \ref{ch2:hdrdx}.
For example, let us consider two of the cases.
\begin{enumerate}
\item $M=(x\;M_1\ldots M_n)$. Let $\tau_i$ be the reduction sequences obtained for $M_i[x:=N]$
by the induction hypothesis. Let $\tau=\tau_1\#\ldots\#\tau_n$. By induction on $|\si|$,
we define the following transformation $\si^\circ$. We eliminate the outermost $\th$-redexes
from $\si$, that is, redexes $R$, where $N\fl^{\si'}R\ra_{\th}R'\fl^{\si''}N'$.
Observe that an outermost $\th$-redex appears in $\si$ iff $\si$ is standard by reason of point 2.
(a) of Definition \ref{ch2:stlmrt}. Let $\si$ be such that $N=\m\a. P \fl^{\si_1} \m\a.[\a]R \ra_{\th}R\fl^{\si_2}N'$,
 where $\si=\si_1\#\si_2$ and $R$ is the first $\th$-redex in $\si_1$.
Let $\xi$ be $(\m\a. P)M_1[x:=N]\ldots M_n[x:=N] \ra_{\mu}^n \m\a. P[\a:=_r M_1[x:=N]]\ldots [\a:=_rM_n[x:=N]]\fl^{\si_1'}
\m\a.[\a](R\;M_1[x:=N]\ldots M_n[x:=N])\ra_{\th}(R\;M_1[x:=N]\ldots M_n[x:=N])$, where $\si_1'$
is obtained from $\si_1$ by Lemma \ref{ch2:sbtrm}. Then let $\si^\circ=\xi\#(\si_2)^\circ$
where $\si=[R]\#\si'$. The reduction sequence
$M[x:=N] \fl^{\si^\circ} (N'\;M_1[x:=N]\dots M_n[x:=N]) \fl^{\tau} (N'\;M_1[x:=N']\dots M_n[x:=N'])$
is appropriate. We prove by induction on $|\si|$ that $|\si^\circ|\leq n+n\cdot(\lan\si\ran_\rho+\lan\si\ran_\theta)$:
$|\si^\circ|=|\xi|+|\si_2^\circ|=1+n+|\si_1|+n\cdot\lan \si_1\ran_{(\rho,\a)}+|\si_2^\circ|\leq 1+n+|\si_1|+n\cdot\lan
\si_1\ran_{\rho}+|\si_2^\circ|\leq 1+n+|\si_1|+n\cdot\lan\si_1\ran_{\rho}+|\si_2|+n\cdot(\lan \si_2 \ran_{\th}+\lan
\si_2 \ran_{\rho})=|\si|+n\cdot(\lan \si \ran_{\th}+\lan \si \ran_{\rho})$.
Then $|\si^\circ+\tau|=|\si^\circ|+\smash{\sum_{i=1}^{i=n}|\tau_i|}\leq |\si|+n\cdot(\lan\si\ran_{\th}+\lan \si\ran_{\rho})+|\tau_i|\cdot
\smash{\sum_{i=1}^{i=n}}|M_i|_x+\smash{\sum_{i=1}^{i=n}}sumarg(M_i,x)\cdot(\lan\si\ran_{\th}+\lan \si\ran_{\rho})=
|\si|\cdot|M|_x+sumarg(M,x)\cdot(\lan \si\ran_{\th}+\lan \si\ran_{\rho})$.
\item $M=(\l y.M_1)M_2\ora{P}$. The induction hypothesis gives $\tau_i$
such that $M_i[x:=N]\fl^{\tau_i}M_i[x:=N']$ $(1 \leq i\leq ,k)$. Then we can choose $\tau
=\tau_1\# \ldots\#\tau_k$. \qedhere
\end{enumerate}
\end{proof}

Lemmas \ref{ch2:comp} and \ref{ch2:comput} combine the results of the preceding lemmas: we substitute in place of a variable common in the members of a standard reduction sequence a new term such that we are also equipped with a standard reduction sequence starting from it. In the case of the $\l$-substitution we obtain an inequality for the length of the new standard reduction sequence, whereas in the case of the $\mu$-substitution we have an exact estimation.

\begin{lem} \label{ch2:comp}
Let $M\fl^{\si}M'$ and $N\fl^{\nu}N'$ be standard. Then there is a standard reduction $\tau$ such that
$M[x:=N]\fl^{\tau}M'[x:=N']$ and $|\tau |\leq |\si |+|M'|_x \cdot
|\nu | + sumarg(M',x)\cdot (\lan\nu\ran_{\th}+\lan\nu\ran_{\rho})$.
\end{lem}

\begin{proof}
The proof goes by induction on $( |\si |,comp(M))$,
taking into account the various points of Definition
\ref{ch2:stlmrt}.
The case $|\si |=0$ is treated by Lemma \ref{ch2:cmp}.
Let $\si=[R]\# \si'$. We treat some of the typical cases.
\begin{enumerate}
\item $M=\m\a. M_1$.
 If $M_1\fl^{\si}M_1'$, then the induction hypothesis applies.
Let $\si$ be standard by reason of point 2 (a) of Definition \ref{ch2:stlmrt}.
Let $M\fl^{\si_1}\m\a.[\a]M_2 \ra_{\th}M_2\fl^{\si_2}M'$.
Lemma \ref{ch2:sbtrx} and the induction hypothesis give standard $\nu_1$ and $\nu_2$ such that
$M[x:=N]\fl^{\nu_1}\m\a.[\a]M_2[x:=N])\ra_{\th}M_2[x:=N]\fl^{\nu_2}M'[x:=N']$.
Moreover, since $\m\a.[\a]M_2[x:=N]$ is the first $\th$-redex in the sequence,
$\nu=\nu_1\#[\m\a.[\a]M_2[x:=N]]\#\nu_2$ is standard by virtue of Definition \ref{ch2:stlmrt}.
The case of point 2. (b) of Definition \ref{ch2:stlmrt} follows from the induction hypothesis.
\item $M=(\m\a. M_1)M_2\ldots M_n$. Assume $M\ra_{\m}(\m\a. M_1[\a:=_rM_2]\ldots M_n)\fl^{\si'}M'$.
Then $((\m\a. M_1)M_2)[x:=N]$ is the head redex of $M[x:=N]$ and Lemma \ref{ch2:sthdrdx}
together with the induction hypothesis yield the result. If $(\m \a. M_1)M_2$ is not involved in $\si$,
then the induction hypothesis applies.
\item $M=(x\;M_2\ldots M_n)$. The proof is analogous to that of Lemma \ref{ch2:cmp}.
We define, by induction on $\nu$, a standard reduction sequence $\nu^\circ$ in the same way
as in Lemma \ref{ch2:cmp}. We let $\tau=\nu^\circ\#\tau_2\#\ldots\#\tau_n$,
where $\tau_i$ is obtained from $M_i[x:=N]$ $(2\leq i\leq n)$ by the induction hypothesis.
By examining the various cases of Definition \ref{ch2:stlmrt}, we prove by induction on $|\nu|$ that
$\tau\in St$. As to the length of $\tau$, we have
$|\tau| = |\nu^\circ|+|\tau_2|+\ldots+|\tau_n| \leq |\nu|+(n-1)\cdot(\lan\nu\ran_{\th}+\lan \nu \ran_{\rho})+
\smash{{\sum_{i=2}^{n}}}(|\si_i|+|M_i'|_x\cdot|\nu|+sumarg(M_i',x)\cdot(\lan\nu\ran_{\th}+\lan \nu \ran_{\rho}))=
|\si|+|M'|_x\cdot|\nu|+sumarg(M',x)\cdot(\lan \nu\ran_{\th}+\lan\nu\ran_{\rho})$.
\end{enumerate}
The remaining cases are proved analogously.
\end{proof}

\begin{lem} \label{ch2:comput}
Let $M\fl^{\si}M'$ and $N\fl^{\nu}N'$ be standard.
Then there is standard
sequence $\tau$ such that
$M[\a :=_rN]\fl^{\tau}M'[\a :=_rN']$ and $|\tau|= |\si |+\lan \si
\ran_{(\rho , \a)}+|M'|_{\a }\cdot |\nu |$.
\end{lem}

\begin{proof}
The proof goes by induction on $( |\si | , comp(M))$,
similarly to that of the previous lemma. We consider some of the
cases according to Definition \ref{ch2:stlmrt}.
The case $\si$=0 is treated in Lemma \ref{ch2:cmp}.
Let $\si=[R]\# \si'$ for some $\si'$.
\begin{enumerate}
\item $M=[\b]M_1$.  Assume $\b=\a$. If $M_1\fl^{\si}M_1'$ with $M'=[\a]M_1'$,
 then the induction hypothesis applies. Otherwise, let $M\fl^{\si_1}[\a]\m\g. M_2 \ra_{\rho}M_2[\g:=\a]\fl^{\si_2}M'$.
Similarly to the proof of Lemma \ref{ch2:sbtrm}, we have the standard reduction sequence $\tau$:
$M[\a :=_rN] \fl^{\si_1[\a:=_rN]} [\a](\m\g. M_2[\a :=_rN])\;N \ra_{\m} [\a]\m\g. M_2[\a :=_rN][\g:=_rN]) \ra_{\rho} M_2[\g:=\a][\a :=_rN]
\fl^{\tau_2}M'[\a:=_rN']$, where $\tau_2$ is obtained from $\si_2$ by the induction hypothesis and the standardness follows
from Lemma \ref{ch2:sthdrdx2} and the induction hypothesis, where we have made use of the fact that $[\a]\m\g. M_2$ is the first
$\rho$-redex in $\si_1$.  For the length of $\tau$ we have $|\tau|=2+|\si_1|+|\tau_2|=2+|\si_1|+|\si_2|+\lan\si_2\ran_{(\rho,\a)}+
|M'|_{\a}\cdot|\nu|=|\si|+\lan\si\ran_{(\rho,\a)}+|M'|_{\a}\cdot|\nu|$. Assume now $\b\neq \a$. The case when $M$ does not reduce to a
$\rho$-redex or it reduces to a $\rho$-redex but this is not involved in $\si$ is again
obvious. Let $M\fl^{\si_1}[\b]\m\g. M_2$.
Then $M[\a :=_rN]\fl^{\si_1[\a:=_rN]}[\b]\m\g. M_2[\a :=_rN] \ra_{\rho} M_2[\g :=\b][\a :=_rN]\fl^{\tau'}M'[\a :=_rN']$ is standard,
and $\tau'$ is obtained by the induction hypothesis. The equation for the length of $\tau$ is obviously valid in this case, too.
\item $M=(\m\b. M_1)M_2\ldots M_n$. If $(\m\b. M_1)M_2$ is not involved in $\si$, then the induction hypothesis applies.
Otherwise, since $hd(M)=hd(M[\a:=_rN])$, we have the result by Lemma \ref{ch2:sthdrdx}, and again by the induction hypothesis.
\end{enumerate}
All the remaining cases are proved in a similar way.
\end{proof}

\subsection{The standardization theorem for the \texorpdfstring{$\lmrt$}{lambda-mu-rho-theta}-calculus}

We are in a position now to state and prove the standardization theorem for the $\lmrt$-calculus.
As an additional result, we obtain an upper bound for the lengths of the standard $\l\m$I-reduction sequences.
First of all, we harvest the results of the previous subsection in a definition: the definition below assigns
values to pairs formed by redexes and their containing terms. The definition is of technical interest: it makes us possible
to find an upper bound for the standardization of a reduction sequence.

\begin{defi}\label{ch2:mr}
Let $R$ be a redex in a term $M$, the number $m(R,M)$ is defined as follows.
\begin{enumerate}
\item If $R=(\m \a. P) Q$, then $m(R,M)=|P|_{\a}$.
\item If $R=(\l x.P) Q $, then $m(R,M)=2\cdot comp(M)-2$.
\item If $R=[\a]\m \b. P$, then $m(R,M)=1$.
\item If $R=\mu \a. [\a] P$ and $R$ has $n$ arguments in $M$, then $m(R,M)=2n-1$.
\end{enumerate}

\end{defi}

The definition of $m(R,M)$ resembles the corresponding definition applied by Xi \cite{Xi},
where $m(R)$ is the number of the occurrences of $x$ in $P$ provided $R=(\l x.P)Q$.
The additional redexes, however, compel us to change the value of $m(R,M)$ even for the case of the $\b$-redex.
The lemma below will be used in the next subsection.

\begin{lem} \label{ch2:sta0}
If $R$ is a redex in $M$, then $m(R,M)\leq 2\cdot (comp(M)-1)$.
\end{lem}

\begin{proof}
Immediate by Definition \ref{ch2:mr}.
\end{proof}

The following lemma is the main lemma for obtaining the standardization result and the bound for
the standard reduction sequences in Theorem \ref{ch2:stlm}. In what follows, let $|\si|^*=max(|\si|,1)$,
where $\si$ is a reduction sequence.

\begin{lem} \label{ch2:sta}
Let $M\fl ^{\si}M'\ra ^{R}M''$ such that ${\si}$ is a standard reduction sequence.
Then there exists a standard reduction sequence $M\fl ^{\tau}M''$
such that $|\tau |\leq 1+max(m(R,M'),2)\cdot |\si|^*$. Furthermore, if $M$ is a $\l\m$I-term, then $1+|\si|\leq |\tau|$.
\end{lem}

\begin{proof}
The proof goes by induction $( |\si |,comp(M))$. The
case of $|\si|=0$ is obvious, thus we may assume $|\si|>0$. We examine the points of Definition \ref{ch2:stlmrt}.
We treat some of the more interesting cases.
\begin{enumerate}
\item $M=[\a]M_1$. If $M_1\fl^{\si}M_1'$ such that $M'=[\a]M_1'$ and there are no $\rho$-redexes as head redexes
in $\si$ including $M''$, then the induction hypothesis applies. Assume $M\fl^{\si_1}[\a]\m\b. M_2$ such that
 $\si_1\in St$ and $[\a]\m\b. M_2$ is the first $\rho$-redex in the sequence.
Let us suppose, according to point 3 (a) of Definition \ref{ch2:stlmrt}, $[\a]\m\b. M_2 \ra_{\rho}M_2[\b:=\a]\fl^{\si_2}M'$,
where $\si=\si_1\#[[\a]\m\b. M_2]\#\si_2$ and $\si_i\in St$ $(i\in \{1,2\})$. By the induction hypothesis applied to $\si_2$,
we obtain a $\tau'\in St$ such that $|\tau'|\leq 1+max(m(R,M'),2)\cdot|\si_2|^*$. Then let $\tau=\si_1\#[[\a]\m\b. M_2]\#\tau'$.
Hence $|\tau|=1+|\si_1|+|\tau'|\leq 1+|\si_1|+1+max(m(R,M'),2)\cdot|\si_2|^*\leq 1+max(m(R,M'),2)\cdot|\si|^*$.
Assume we have $[\a]\m\b. M_2 \fl^{\si_2} [\a]\m\b. M_2'$ with $M'=[\a]\m\b. M_2'$ and $M_2\fl^{\si_2}M_2'$,
by reason of point 3 (b) of Definition \ref{ch2:stlmrt}. If $R\leq M_2'$, then we obtain the result by the induction hypothesis.
Assume $R=M'$. Then  $\tau=\si_1\#[[\a]\m\b. M_2]\#\si_2'$ is appropriate,
where $\si_2'$ is $M_2[\b:=\a]\fl^{\si_2[\b:=\a]}M_2'[\b:=\a]$. The estimation for $|\tau|$ follows easily,
since $|\si_2'|=|\si_2|$.
Finally, if $M$ is a $\l\m$I-term, the result follows from the induction hypothesis by inspection of the various subcases.
For example, consider the case when $[\a]\m\b. M_2 \ra_{\rho} M_2[\b:=\a]\fl^{\si_2}M'$,
where $\si=\si_1\#[[\a]\m\b. M_2]\#\si_2$, that is, the case described by point 3 (a) of Definition \ref{ch2:stlmrt}.
If $\tau'$ is the standard reduction sequence corresponding to $\si_2$ by the induction hypothesis and
$\tau=\si_1\#[[\a]\m\b. M_2]\#\tau'$, then $1+|\si_2|\leq |\tau'|$ and we obtain the result.
\item $M=(\m\a. M_1)M_2\ldots M_n$. Let $\si$ be standard by virtue of point 5.(a) of Definition \ref{ch2:stlmrt}.
Then $(\m\a. M_1)M_2\ldots M_n\ra_{\mu}(\m\a. M_1[\a:=_rM_2]\ldots M_n)\fl^{\si'}M'$ with $\si'\in St$.
The induction hypothesis applied to $\si'$ provides us with a standard $\tau'$ with appropriate length such that
$(\m\a. M_1[\a:=_rM_2]\ldots M_n)\fl^{\tau'}M''$. By this the result follows.
Assume $\si$ is standard by reason of point 5.(b) of Definition \ref{ch2:stlmrt}.
Then $(\m\a. M_1)M_2\ldots M_n \ra^{\si_1}(\m\a. M_1')M_2\ldots M_n\fl^{\si_2}\ldots\fl^{\si_n}(\m\a. M_1')M_2'\ldots M_n'=M'$,
where $\si=\si_1\#\ldots\#\si_n$. If $R\leq M_i'$, then the induction hypothesis gives the result.
Let $R=(\m\a. M_1')M_2'$. Then $\tau$ can be chosen as
$(\m\a. M_1)M_2\ldots M_n \ra_{\m}(\m\a. M_1[\a:=_rM_2]\ldots M_n)\fl^{\tau_1}
(\m\a. M_1'[\a:=_rM_2']\ldots M_n)\fl^{\si_3}\ldots\fl^{\si_n}
(\m\a. M_1'[\a:=_rM_2']\ldots M_n')$, where $\tau_1$ is obtained from Lemma \ref{ch2:comput}. Moreover,
 $|\tau|=1+|\tau_1|+\sum_{i=3}^n|\si_i|=1+|\si_1|+\lan\si_1\ran_{(\rho,\a)}+|M_1'|_{\a }\cdot|\si_2|+\sum_{i=3}^n|\si_i|\leq 1+
2|\si_1|+|M'|_{\a }\cdot|\si_2|+\sum_{i=3}^n|\si_i|\leq 1+max(m(R,M'),2)\cdot|\si|^*.$
Assume $R=\m\a. [\a]M_1''$ is a $\th$-redex. In this case $\si_1$ is standard by virtue of point 2 (b) (ii)
of Definition \ref{ch2:stlmrt}. Let $\m\a. [\a]M_1^*$ be the first $\th$-redex such that an initial segment
$\si_1'$ of $\si_1$ produces $\m\a. [\a]M_1^*$ starting from $\m\a. M_1$. Let $\si_1=\si_1'\#\si_1''$.
Then $(\m\a. M_1)M_2\ldots M_n \ra_{\m}^{n-1}\m\a. M_1[\a:=_rM_2]\ldots [\a:=_rM_n]\fl^{\tau_1}
\m\a.[\a](M_1^*\;M_2\ldots M_n) \ra_{\th}(M_1^*\;M_2\ldots M_n)\fl^{\si_1''}(M_1'\;M_2\ldots M_n)
\fl^{\si_2}\ldots\fl^{\si_n}(M_1'\;M_2'\ldots M_n')=M'$ is standard, where $\tau_1$ is obtained from $\si_1'$
by Lemma \ref{ch2:sbtrm}. As to the length of $\tau$, we have $|\tau|=1+(n-1)+|\tau_1|+|\si_1''|+
\sum_{i=2}^n|\si_i|=1+(n-1)+|\si_1'|+(n-1)\cdot\lan\si_1'\ran_{(\rho,\a)}+|\si_1''|+\sum_{i=2}^n|\si_i|\leq 1+
|\si|+(n-1)\cdot(1+|\si|)=1+n\cdot|\si|+(n-1)\leq 1+max(m(R,M'),2)\cdot|\si|^*$. When $M$ is a $\l\m$I-term, we obtain
the result by the induction hypothesis. Let us only treat the last case, where
$(\m\a. M_1) M_2\ldots M_n \ra^{\si_1}(\m\a. M_1')M_2\ldots M_n \fl^{\si_2}\ldots\fl^{\si_n}(\m\a. M_1')M_2'\ldots M_n'=M'$
and $R=\m\a. M_1'=\m\a.[\a]M_1''$. If $\m\a.[\a]M_1^*$ is the first $\th$-redex in $\si_1$ such that
$\si_1=\si_1'\#\si_1''$ and $\tau_1$ is obtained from $\si_1'$ by Lemma \ref{ch2:sbtrm} and $\tau$ is defined as above,
then $1+|\si|=1+|\si_1'|+|\si_1''|+\sum_{i=2}^{n}|\si_i|\leq |\tau|=(n-1)+|\tau_1|+1+|\si_1''|+\sum_{i=2}^{n}|\si_i|$,
where $|\si_1'|\leq |\tau_1|$ by Lemma \ref{ch2:sbtrm}.\qedhere
\end{enumerate}
\end{proof}

\begin{defi}
Let $\si$ be the reduction sequence $M_1\ra^{R_1}M_2\ra^{R_2}\dots
\ra^{R_n}M_{n+1}$. Denote by $\ma{M}(\si )$ (the measure of \,$\si$)
the number $\prod^{n}_{i=1}(1+max(m(R_i,M_i),2))$.
\end{defi}

\begin{thm} \label{ch2:stlm}
Let $\si$ be the reduction sequence $M=M_1\ra^{R_1}M_2\ra^{R_2}\dots
\ra^{R_n}M_{n+1}$. Then there is a standard reduction
sequence $st(\si )$ such that $M_1\fl ^{st(\si )}M_{n+1}$ and
$|st(\si )|\leq  \ma{M}(\si )$. Moreover, if $M$ is a $\l\m$I-term, then $|\si|\leq |st(\si)|$ also holds.
\end{thm}

\begin{proof}
The statement of the theorem is proved by induction on $|\si |$.
\begin{enumerate}
\item If $|\si |=1$, then our claim follows directly from Lemma \ref{ch2:sta}.
\item Let $\si =\si'\# [R_{n}]$, where $|\si'|\geq 1$.
By the induction hypothesis, we can find a standard
$st(\si')$ with appropriate length such that
$M_1\fl^{st(\si')}M_n$. Moreover, $|st(\si')|^*=|st(\si')|$.
Then, by Lemma \ref{ch2:sta}, there is a standard $M_1\fl^{\tau}M_{n+1}$ such
that $|\tau |\leq 1+max(m(R_n,M_n),2)\cdot |st(\si')|^*\leq
(1+max(m(R_n,M_n),2))\cdot |st(\si')|^*$, which yields the
result.\qedhere
\end{enumerate}
\end{proof}

\begin{thm}\label{ch2:lftmtheorem}
If $M$ is a $\l\m$I-term, then a standard reduction sequence starting from $M$ and leading to the normal form of $M$
is the leftmost reduction sequence and it is a reduction sequence of maximal length.
\end{thm}

\begin{proof}
Let $M$ be a $\l\m$I-term. Assume $M\fl^{\si}M'$ where $M'$ is the normal form of $M$. The proof goes
by induction on $( |\si|, comp(M))$.
We may assume $M\in HNF$. Otherwise, by Lemma \ref{ch2:hdrdxlem}, the head redex of $M$ is involved in $\si$,
then Lemma \ref{ch2:invhdrdx} yields that $\si=[hd(M)]\#\si'$. That is, if $M\notin HNF$, then the induction
hypothesis applies. Let $M=(x\;M_2\ldots M_n)$. By Definition \ref{ch2:stlmrt} there exist
$\si_i\in St$ $(2\leq i\leq n)$ such that $\si=\si_2\#\ldots\#\si_n$ and
$(x\;M_2\ldots M_n)\fl^{\si_2}(x\;M_2'\ldots M_n)\fl^{\si_3}\ldots\fl^{\si_n}(x\;M_2'\ldots M_n')$.
Then the induction hypothesis applied to $\si_i$ $(2\leq i\leq n)$ gives the result.
The leftmost reduction has a maximal length by Theorem \ref{ch2:stlm}.
\end{proof}

\section{The estimation for the lengths of the reduction sequences of the \texorpdfstring{$\lmrt$}{lambda-mu-rho-theta}-calculus}\label{section:3}

In this section we present an application of Theorem \ref{ch2:stlm} which, through the standardization, provides us
with a bound for the length of the standard reduction sequence. Making use of the fact that, by Theorem \ref{ch2:lftmtheorem},
the standard reduction sequence for a $\lm$I-term is unique and the Church-Rosser property is valid for the $\lm$I-calculus,
it does not make a difference which normalizing reduction sequence we start from and obtain its standardization. Hence, we choose a
normalization sequence $\si$ the measure of which, $\mathcal{M}(\si)$, can easily be estimated, which is, at the same time, an upper
bound for the standardization of $\si$. By Theorem \ref{ch2:stlm}, we have thus obtained an upper bound for $|\si|$.
We extend this result to the general case by finding a translation $[\![M]\!]_k$ of $M$ with an appropriate $k$,
where $[\![M]\!]_k$ is a $\lm$I-term such that lengths of the types of the redexes in $M$ is the same as those
of $[\![M]\!]_k$ and $\eta(M)\leq \eta([\![M]\!]_k)$ and the complexity of $[\![M]\!]_k$ is bounded by a linear function of
the complexity of $M$.

\subsection{The estimation the lengths of the reduction sequences of the \texorpdfstring{$\lmrt I$}{lambda-mu-rho-theta I}-calculus} \label{ch2:lmri}

In this subsection we give an estimation for the lengths of the
reduction sequences in the $\lmrt$I-calculus. To this end we define a normalization strategy such that the lengths of
reduction sequences obeying that strategy can be assessed easily and we can even establish bounds for the sizes of
the developments. Prior to this, we need the rank of a redex. Intuitively, the rank of a redex is the length of type of
the $\l$- or $\mu$-abstraction of the redex. This is exactly the quantity that can decrease by a reduction.

\begin{defi}\label{ch1:lhtype}
\hfill
\begin{enumerate}
\item The rank of a redex $R$ in a term $M$ is defined as follows.
\begin{itemize}
\item If $R=(\l x.M_1)M_2$, then $rank(R,M) =lh(type(\l x.M_1))$.
\item If $R=(\mu \a. M_1) M_2$, then $rank(R,M) =lh(type(\mu \a. M_1))$.
\item If $R=[\a]\mu \b. M$, then $rank(R,M) =lh(type(\mu \b. M))$.
\item If $R=\m \a. [\a]M$, then $rank(R,M) =lh(type(\mu \a. (\a\;M)))$.
\end{itemize}
\item The rank of a term $M$ is $rank(M)=max\{rank(R,M)\;|\;R\textrm{ is a redex in }M\}$.
\item Define $NF_k=\{M\;|\;rank(M)<k\}$.
\end{enumerate}
\end{defi}

The following lemma states that reductions do not increase the rank.

\begin{lem}\label{ch1:sbstrn}\label{ch1:typm}\label{ch2:rank}
Let $M$,$N$ be terms.
\begin{enumerate}
\item We have $rank(M[x:=N])\leq max\{rank(M),rank(N),lh(type(x))\}$ and

$rank(M[\a :=_rN])\leq max\{ rank(M), rank(N), lh(type^*(\a ))\}$,\\
where $type^*(\a )=A$ if $type (\a )=\neg A$.
\item If $M\fl M'$, then $rank(M)\geq rank(M')$.
\end{enumerate}
\end{lem}

\begin{proof}
\begin{enumerate}
\item By induction on $comp(M)$.
\item It is enough to prove if $M\ra^{R}M'$, then $rank(M)\geq rank(M')$. The proof goes by
induction on $comp(M)$ and we use the first item.\qedhere
\end{enumerate}
\end{proof}

We are now in a position to define the notion of a $k$-reduction sequence, which will denote a specific normalization strategy in what follows.

\begin{defi}\label{ch2:dgood}\label{def:ngood}
\hfill
\begin{enumerate}
\item We say that a reduction sequence $\nu$ is a $k$-reduction
sequence, if every redex in $\nu$ is of rank $k$.
\item A reduction sequence $M\fl^\nu M'$ is a $k$-normalization for a given term $M$, if it is a
$k$-reduction sequence and $M'\in NF_k$.
\item
A reduction sequence $\xi$ starting from a term is good, if, at each reduction step, it chooses the
leftmost, innermost redex of maximal rank, that is, the redex containing no
other redexes of maximal rank and stands in the leftmost position
among these redexes.
\end{enumerate}
\end{defi}

Let $\si$ be a good reduction sequence starting from $M$, assume $rank(M)=k$. Then $\si$ starts with the leftmost,
innermost redex of rank $k$ and chooses the leftmost, innermost redex of maximal rank every time.
Since $M$ is strongly normalizable, $\si$ is necessarily finite. By Lemma \ref{ch2:rank},
the ranks of the redexes involved in $\si$ form a monotone decreasing sequence.
Thus, if $\si$ is a good normalizing sequence, then the sequence of redexes of rank $k$ in $\si$ comes
to an end and $\si$ continues with a leftmost, innermost redex of maximal rank, which is less than $k$. Hence, $\si$ is the concatenation
of $l_i$-normalization sequences ($ 1 \leq i \leq s$) with $l_1 = k > l_2 > ... > l_s \geq 1$.

The next two lemmas show that good k-normalization sequences can be dissected easily so that we
are able to estimate their lengths in the proof of Lemma \ref{ch2:sizet}.

\begin{lem} \label{ch2:sep}\label{ch2:sepl}~
\hfill
\begin{enumerate}
\item Let $rank((\m \a. P) Q)=k$ and $x\notin Fv(P)$.
If $(\m \a. P) Q {\fl}^{\nu} U$ and $\nu$ is a good $k$-normalization
sequence, there are terms $P',Q',U'$ and good $k$-normalization sequences ${\nu}_1, {\nu}_2,
\nu_3$ such that $P{\fl}^{\nu_1} P'$, $Q{\fl}^{\nu_2}
Q'$, $(\m\a.P') x{\fl}^{\nu_3} U'$, $U=U'[x=Q']$ and $\nu
=\nu_1\#\nu_2\# \nu_3[x:=Q']$.
\item Let $rank((\l y. P) Q)=k$ and $x\notin Fv(P)$.
If $(\l y. P) Q{\fl}^{\nu} U$ and $\nu$ is a good $k$-normalization
sequence, there are terms $P',Q',U'$ and good
$k$-normalization sequences ${\nu}_1, {\nu}_2$ such that
$P{\fl}^{\nu_1} P'$, $Q{\fl}^{\nu_2} Q'$, $(\l y.P')
x{\ra}^{\nu_3} P'[y:=x]=P''$, $U=P''[x:=Q']$ and $\nu =\nu_1\#
\nu_2\# \nu_3[x:=Q']$.
\end{enumerate}
\end{lem}

\begin{proof}
\begin{enumerate}
\item The algorithm proceeds by eliminating the innermost $k$-redexes
from left to right, that is we have (possibly empty) $\nu_1$ and
$\nu_2$- both being $k$-normalization sequences such that
$\nu_1\# \nu_2$ is an initial subsequent of $\nu$ and $P\fl^{\nu_1}P'\in {NF}_k$,
$Q\fl^{\nu_2}Q'\in {NF}_k$. Then $\nu$ continues with reducing
$(\m\a. P') Q'$ and the redexes created by this reduction. It is
immediate to check that when reducing $(\m \a. P') Q'$, the created $k$-redexes
can only be redexes of the
form $(\l y.V[\a :=_rQ']) Q'$ for some $\l y.V$ of rank $k$ such
that $[\a]\l y.V \leq P'$, so for every $k$-redex $R$ in $\m \a. P'[\a :=_rQ']$
there is an $R'$ in $\m \a. P'[\a :=_rx]$ such that $R=R'[x:=Q']$. Reducing with these
$\b$-redexes in $\m \a. P'[\a :=_rQ']$, no more $k$-redexes are
created. This proves our assertion.
\item Analogous to the first point. \qedhere
\end{enumerate}
\end{proof}

\begin{lem}\label{ch2:sepnf}
Let $rank((\m \a. P) x)=k$, $\m \a. P\in {{NF}}_k$ and
$x\notin Fv(P)$. If $(\m \a. P)x {\fl}^{\nu} U$,
$\nu$ is a good $k$-normalization sequence,
and $U\in {{NF}}_k$,
then $|\nu | \leq comp(P)$ and $comp(U) \leq 2\cdot comp(P)$.
\end{lem}

\begin{proof}
Since $\m \a. P\in {{NF}}_k$, in $\m \a. P[\a :=_rx]$ $k$-redexes of the form
$(\l y.Q[\a :=_rx])x$ can only occur, where $[\a]\l y.Q \leq P$
and $rank(\l y.Q)=k$. Subsequently reducing these redexes gives
$U$, which means that $U$ can be obtained in at most
$|P|_{\a}+1\leq comp(P)$ steps.
Considering the above argument, since $x$ is a variable, the $\b$-reduction steps in
$\nu$ does not increase the size of the term, so $comp(U)\leq
comp(\m \a. P[\a :=_rx])=1+comp(P)+|P|_{\a }\leq 2\cdot comp(P)$.
\end{proof}

The lemma below gives estimations for good $k$-normalization sequences. We may observe that
the obtained bounds does not depend on $k$.

\begin{lem}\label{ch2:sizet}
Let $M$ be a term such that $rank(M)=k$. If $M\fl^{\nu} M'$ and $\nu$ is a
good $k$-normalization sequence, then $comp(M')\leq 2^{comp(M)-1}$
and  $|\nu |\leq 2^{comp(M)-1}$.
\end{lem}

\begin{proof}
The proof of $comp(M')\leq 2^{comp(M)-1}$ goes by induction on $comp(M)$.
\begin{enumerate}
\item The case $M=x$ or $M=\l x.M_1$ is obvious.
\item Let $M=\m \a. M_1$.
\begin{enumerate}
\item If $M=\m \a.[\a]M_1$ is a $\th$-redex of rank $k$, then,
since the algorithm eliminates $k$-redexes from bottom to up and
from left to right, we have a $\nu'\leq
\nu$ such that $\mu \a.[\a]M_1 \fl^{\nu'}\mu \a.[\a]M_1'
\ra^{R}M_1'=M'$. But in this case $M\ra_{\th}M_1\fl^{\nu'}M'$
is valid as well, thus by the induction hypothesis
$comp(M')\leq 2^{comp((M_1)-1)}<2^{comp(M)-1}$.
\item If $\m \a. M_1$ is not a $\th$-redex, but reduces to a $\th$-redex
of rank $k$ in the course of the process, then a reasoning
analogous to the above one works.
\item If $\m \a. M_1$ is not a
$\th$-redex and it neither reduces to a $\th$-redex, then the
induction hypothesis applies.
\end{enumerate}
\item Let $M=(M_1\;M_2)$.
\begin{enumerate}
\item If $M$ is not a $k$-redex, then we prove that $M$ cannot reduce to a $k$-redex.

Suppose on the
contrary that there is some initial subsequent of $\nu$ such that
it reduces $M$ to a $k$-redex, take $\nu'$ as the shortest such
reduction sequence. Suppose $M$ reduces to a $\m$-redex (the case
of a $\b$-redex is similar). In this case we have $M\fl^{\nu'}(\m
\b. N_1) N_2$, where $M_1\fl \m \b. N_1$ and $M_2\fl N_2$. Then
$M\fl^{\nu''}(N_3\;N_2)\ra^{R'}(\m \b. N_1) N_2$ must hold for
some $R'$, $\nu''$ such that $\nu'=\nu''\# [R']$ and for some
$N_3$, $N_3$ not beginning with a $\m$. This means $N_3=R'$ would
be again a $k$-redex, but a straightforward examination of the
possible cases shows it is impossible.

Hence we have
$M'=(M_1'\;M_2')$, $\nu =\nu_1\#
\nu_2$ for some $k$-reduction sequences sequences $\nu_1$, $\nu_2$ and
$M_i\fl^{\nu_i}M_i'$ $(i\in \{ 1,2\})$. Thus by the induction
hypothesis
$comp(M')= comp(M_1')+comp(M_2')\leq
2^{comp(M_1)-1}+2^{comp(M_2)-1} \leq  2^{comp(M)-1}$.

\item If $M$ is a $k$-redex and $M=(\m \a. M_1) M_2$, then $M$ is involved in $\nu$ as a
$\mu$-redex. By Lemma
\ref{ch2:sep}, we have $M_1'$, $M_2'$, $M''$ and $\nu_1$, $\nu_2$,
$\nu_3$ such that
$M_1{\fl}^{\nu_1} M_1'$, $M_2{\fl}^{\nu_2} M_2'$,
$(\m \a. M_1') x{\fl}^{\nu_3} M''$,
$M'=M''[x:=M_2']$ and $\nu =\nu_1\# \nu_2\# \nu_3[x:=M_2']$, provided
$x\notin Fv(M_1)$. From this, by Lemma \ref{ch2:sepnf} and by the
induction hypothesis,
$comp(M')=comp(M''[x:=M_2'])
=comp(M'')+|M''|_x\cdot (comp(M_2')-1)
<comp(M'')\cdot comp(M_2')\leq 2\cdot comp(M_1')\cdot
comp(M_2') \leq  2\cdot 2^{comp(M_1)-1}\cdot
2^{comp(M_2)-1}<2^{comp(M)-1}$.
The case $M=(\l x.M_1)M_2$ is similar.
\end{enumerate}
\item Let $M=[\a]M_1$.
\begin{enumerate}
\item If $M$ does not reduce to a $k$-redex, then the result is obvious.
\item If $M$ is either a $k$-redex, or reduces to a $k$-redex, then
there is a $\nu'$ and a $\m \b. M_2\in {{NF}}_k$ such that
$[\a]M_1 \fl^{\nu'}[\a]\m \b. M_2 \ra^R M_2[\b :=\a ]$ and
$\nu'\# [R]=\nu$. The induction hypothesis for $M_1$ gives the
result.
\end{enumerate}
\end{enumerate}

We prove $|\nu |\leq
2^{comp(M)-1}$ by induction on $comp(M)$. The only interesting case is when $M$ is a redex
of rank $k$. Let, for example, $M=(\mu \a. M_1)M_2$. Since $\nu$
is a $k$-normalization sequence we can assume again that $M$ is
involved in $\nu$. By Lemma \ref{ch2:sep}, we have $M_1'$, $M_2'$
and $k$-normalization sequences ${\nu}_1, {\nu}_2,
\nu_3$ such that $M_1{\fl}^{\nu_1} M_1'$,
$M_2{\fl}^{\nu_2} M_2'$, $(\m \a. M_1)'x{\fl}^{\nu_3} M''$,
$M'=M''[x=M_2']$ and $\nu =\nu_1\# \nu_2\#
\nu_3[x:=M_2']$, provided $x\notin Fv(M_1)$. Then, using
Lemma \ref{ch2:sepnf} and the induction hypothesis, we obtain
$|\nu | = |\nu_1|+|\nu_2|+|\nu_3[x:=M_2']|=|\nu_1|+|\nu_2|+|\nu_3|
\leq 2^{comp(M_1)-1}+2^{comp(M_2)-1}+2^{comp(M_1)-1} =
2^{comp(M_1)}+2^{comp(M_2)-1}\leq 2^{comp(M)-1}.$
\end{proof}

\begin{defi} Let $\emph{tower}$ defined by
$\emph{tower}(n,m)=\left\{ \begin{array}{ll} m & \tx{ if
}n=0,\\ 2^{{\emph{tower}}(n-1,m)} &
\tx{ if }n>0.\end{array}\right.$

In other words, the integer $\emph{tower}(n,m)$ is $2^{\cdot^{\cdot^{\cdot^{2^{m}}}}}$, where $2$ is repeated $n$ times.
\end{defi}

\begin{thm}\label{ch2:msg}
Let $M$ be a term such that $rank(M)=k$. If $M\fl^{\si}N$,
$\si$ is a good reduction sequence and $N\in
NF$, then $\ma{M}(\si )<\emph{tower}(k+1,comp(M))$.
\end{thm}

\begin{proof}
We first prove by induction on $k$ that\\
$\ma{M}(\si )<{\textrm{tower}}(1,{\textrm{tower}}(1,comp(M))+
\sum_{i=2}^{k}{\textrm{tower}}(i,comp(M)-1))$.
\begin{enumerate}
\item If $k=1$, then $\si$ is a 1-normalization sequence. Suppose $\si$
is $M=M_1\ra^{R_1}M_2\ra^{R_2}\dots \ra^{R_{n-1}}M_n\ra^{R_n}M_{n+1}$ for
some $n\geq 1$. We have, by Lemmas \ref{ch2:sta0} and \ref{ch2:sizet}, $1+max(m(R_i,M_i),2)\leq 2\cdot
comp(M_i)-1\leq 2\cdot 2^{comp(M)-1}-1<2^{comp(M)}$, then
$\ma{M}(\si )=\prod^{n}_{i=1}(1+max(m(R_i,M_i),2))\;<$\\
$\; \prod^{n}_{i=1}2^{comp(M)} =2^{n\cdot comp(M)}$.
Again, by Lemma \ref{ch2:sizet}, we obtain $n=|\si |\leq
2^{comp(M)-1}$, so $\ma{M}(\si )<2^{comp(M)\cdot
2^{comp(M)-1}}\leq 2^{2^{comp(M)}}=$\\
$\textrm{tower}(1,\textrm{tower}(1,comp(M)))$.

\item Let $rank(M)=k+1$ and $k\geq 1$. Assume
$M\fl^{\si'}M'\fl^{\si''}N\in {NF}$, where $\si'$ is a
$k+1$-normalization sequence starting from $M$. By the induction
hypothesis, we have
$\ma{M}(\si '')<{\textrm{tower}}(1,{\textrm{tower}}(1,comp(M'))+
\sum_{i=2}^{k}{\textrm{tower}}(i,comp(M')-1))$.
As above, we obtain again $\ma{M}(\si')<2^{2^{comp(M)}}$. Then,
using the multiplicity of $\ma{M}$ and Lemma
\ref{ch2:sizet}, we can assert
$\ma{M}(\si ) =
\ma{M}(\si')\cdot \ma{M}(\si'')<
2^{2^{comp(M)}}\cdot{\textrm{tower}}\Big(1,{\textrm{tower}}(1,comp(M'))\;+\;
\sum_{i=2}^{k}{\textrm{tower}}(i,comp(M')-1)\Big) <
2^{2^{comp(M)}}\cdot {\textrm{tower}}\Big(1 ,{\textrm{tower}}(1,\textrm{tower}(1,comp(M)-1)) {} +$ \\
$\sum_{i=2}^{k}{\textrm{tower}}(i,\textrm{tower}(1,comp(M)-1))\Big) =$\\
$2^{2^{comp(M)}}\cdot 2^{\underbrace{{2^{2^{comp(M)-1}}+\dots +
2^{\cdot^{\cdot^{\cdot^{2^{comp(M)-1}}}}}}}_k} =$\\
${\textrm{tower}}(1,{\textrm{tower}}(1,comp(M))+ \sum_{i=2}^{k+1}{\textrm{tower}}(i,comp(M)-1))$.

\end{enumerate}
Finally, we prove, by induction on $k$, that\\
${\textrm{tower}}(1,comp(M))+
\sum_{i=2}^{k}{\textrm{tower}}(i,comp(M)-1)\leq
\textrm{tower}(k,comp(M))$.\\
The case $k=1$ is obvious.
Let $k=n+1$ and $n\geq 1$. Applying the induction hypothesis, we
obtain
${\textrm{tower}}(1,comp(M))+
\sum_{i=2}^{n+1}{\textrm{tower}}(i,comp(M)-1) {} =$\\
$\;\underbrace{{2^{comp(M)}+2^{2^{comp(M)-1}}+\dots
+2^{\cdot^{\cdot^{\cdot^{2^{comp(M)-1}}}}}}}_{n+1} {} \leq$\\
$\textrm{tower}(n,comp(M))+\textrm{tower}(n+1,comp(M)-1)
{} <\;\textrm{tower}(n+1,comp(M))$.
\end{proof}

\begin{cor}\label{ch2:bfi}
Let $M$ be a $\l \m I$-term of rank $k$. Every reduction
sequence starting from $M$ has length less than $\twe
(k+1,comp(M))$.
\end{cor}

\begin{proof}
Let $N$ be the normal-form of $M$. By Definition
\ref{def:ngood} and Theorem \ref{ch2:msg}, there exists a $\si$
such that $M\fl^{\si}N$ and $\ma{M}(\si )<\tw (k+1,comp(M))$. By
Theorem \ref{ch2:stlm}, there is a standard $\si'$ such
that $M\fl^{\si'}N$ and $|\si'|<\ma{M}(\si )$. The result follows
now from Theorem \ref{ch2:lftmtheorem}.
\end{proof}

\subsection{Some properties of the function \texorpdfstring{$\eta$}{eta}}\label{ch2:subseceta}

In the next subsection we undertake the task of estimating the lengths of reduction sequences
starting from an arbitrary term by transforming the starting term into a $\lm$I-term and estimating
an upper bound for the reduction sequences of the $\lm$I-term. In order to make the estimation work,
we have to prove that the longest reduction sequences of the transformed terms are at least as long
as those of the original terms. To this end, we perform some calculations concerning
longest reduction sequences of terms and their reducts. This subsection prepares the treatment of
the general case. The lemmas of the subsection compare the lengths of the longest reduction sequences starting
from a redex and from one of its reducts.

\begin{lem}\label{ch2:lhm}
Let $M,N$  and $\ora{P}$ be $\l \m I$-terms. If $\a \notin Fv(N)$, then \\
$\eta ((\m \a. \langle M,[\a]z\rangle)\ora{P})+ \eta (N)\leq \eta ((\m \a. \langle M,[\a](z\;N)\rangle)\ora{P})$.
\end{lem}

\begin{proof}
Let $U=(\m \a. \langle M,[\a]z\rangle)\ora{P}$, $V=(\m
\a. \langle M,[\a](z\;N)\rangle)\ora{P}$. If $\ora{P}$ is empty, the result
is trivial, so may assume $\ora{P}$ is not empty and its
components are $M_1,\dots ,M_n$. We are going to prove if
$U\fl^{\si_1}U'$, $N\fl^{\si_2}N'$ for some $\si_1$, $\si_2$,
$U'$, $N'$, then we have a reduction sequence $\nu$ of $V$ such
that $|\si_1|+|\si_2|\leq |\nu|$. By the second part of Theorem
\ref{ch2:stlm}, it is enough to restrict our attention to the case
when $\si_1$ and $\si_2$ are standard. We may assume that the
head-redex of $U$ is involved in $\si_1$, otherwise the result is
trivial. Furthermore, we may suppose that $\m \a. \langle M,[\a]z\rangle $
is reduced in $|\si_1|$ with all of its arguments
$M_1,\dots ,M_n$. Then $\si_1$ is of the form\\
$U \fl^{\xi} \m \a. \langle M[\a
:=_rM_1]\dots [\a :=_rM_n],[\a](z\;M_1\dots M_n)\rangle$\\
$\fl^{\zeta} \m \a. \langle M', [\a](z\;M_1\dots M_n)\rangle \;\fl^{\zeta^*}\;
\m \a. \langle M',[\a](z\;M_1'\dots M_n')\rangle$,\\
where $M[\a :=_rM_1]\dots [\a :=_rM_n]\fl^{\zeta}M'$
and $\zeta^*=\zeta_1\# \dots \# \zeta_n$ with
$M_i\fl^{\zeta_i}M_i'$ for $1\leq i\leq n$. Let $\xi'$ be
$V\fl^{\xi'}\m
\a. \langle M[\a :=_rM_1]\dots [\a :=_rM_n],[\a](z\;N\;M_1\dots M_n)\rangle
$, then choosing $\nu$ as
$\nu =\xi'\# \zeta \# \si_2 \# \zeta^*$ is appropriate.
\end{proof}

\begin{lem}\label{ch2:etr}\label{ch2:etr2}
Let $M=(\l x. M_1) M_2\ora{P}$ and $N=(M_1[x:=M_2]\;\ora{P'})$.
\begin{enumerate}
\item If $x\in Fv(M_1)$ and $N$ is strongly normalizable, then $M$ is also strongly normalizable
and $\eta (M)= \eta (N)+1$.
\item If $x\notin Fv(M_1)$ and $N,M_2$ are strongly normalizable, then $M$ is also strongly normalizable
and $\eta (M)= \eta (N)+\eta (M_2)+1$.
\end{enumerate}
\end{lem}

\begin{proof}
\begin{enumerate}
\item Let $M\fl^{\si}U$ be an arbitrary reduction sequence, we are
going to show that $|\si |\leq \eta (N)+1$, from which
the result follows. We may suppose that $(\l x. M_1) M_2$ is
involved in $\si$. Then $\si$ is of the following form for some
$\si_1$ and $\si_2$,
$M=(\l x. M_1) M_2\ora{P}
\fl^{\si_1}M'=(\l x. M_1') M_2'\ora{P'} \ra
(M_1'[x:=M_2']\;\ora{P'})\fl^{\si_2}U$
where
$M_i\fl^{\nu_i}M_i'$ $(i\in \{1,2\})$,
$\ora{P}\fl^{\nu_3}\ora{P'}$ and $\si_1=\nu_1\# \nu_2\#
\nu_3$. Let $\si'$ denote the reduction sequence
$M=(\l x. M_1) M_2 \ora{P} \ra
N=(M_1[x:=M_2] \ora{P})\fl^{\si^*}U$, where
$\si^*=\nu_1'\# \nu_2'\# \nu_3\# \si_2$ and $\nu_1'$ is constructed from
$\nu_1$ by Lemma \ref{ch2:sbtrx} with $M_1\fl^{\nu_1}M_1'$ and $M_2$ and $\nu_2'$ is obtained by applying
Lemma \ref{ch2:cmp} to $M_1$ and $M_2\fl^{\nu_2} M_2'$.
Then $|\si |\leq \eta (N)+1$, which is the desired result.
\item Let $M\fl^{\si}U$ be an arbitrary reduction sequence, it is enough to show that $|\si |\leq \eta (N)+ \eta(M_2)+1$.
We may suppose that $(\l x. M_1) M_2$ is
involved in $\si$. Then $\si$ is of the form
$M=(\l x. M_1) M_2\ora{P}
\fl^{\si_1}M'=(\l x. M_1') M_2'\ora{P'} \ra_{\beta}
(M_1'\;\ora{P'})\fl^{\si_2}U$
where
$M_i\fl^{\nu_i}M_i'$ $(i\in \{1,2\})$,
$\ora{P}\fl^{\nu_3}\ora{P'}$ and $\si_1=\nu_1\# \nu_2\#
\nu_3$. $\si$ can obviously be rearranged as
$M\fl^{\nu_2}(\l x. M_1) M'_2\ora{P} \ra_{\beta}N=(M_1\; \ora{P})\fl^{\nu_1\#\nu_3}(M'_1\; \ora{P'})$,
which yields the result. \qedhere
\end{enumerate}
\end{proof}

\begin{lem} \label{ch2:erh}\label{ch2:eth}~
\begin{enumerate}
\item Let $M=[\a]\m \b. M_1$ and $N=M_1[\b :=\a ]$. If $N$ is strongly normalizable,
then $M$ is also strongly normalizable and $\eta (M)=\eta (N)+1$.
\item Let $M=\m \a.[\a]M_1$ be a $\th$-redex. If $M_1$ is
strongly normalizable, then $M$ is also strongly normalizable and
$\eta (M)=\eta (M_1)+1$.
\end{enumerate}
\end{lem}

\begin{proof}
\begin{enumerate}
\item Assume $\si$ is a reduction sequence starting from $[\a]\m \b.M_1$.
We prove $|\si| \leq \eta (N)+1$, from which the result
follows. Let $\si =[R]\# \si'$ for some $\si'$. We distinguish the
various cases according to the form of $\si$.
\begin{enumerate}
\item If $[\a]\m \b. M_1\ra^R_{\rho}M_1[\b :=\a ]\fl^{\si'}M_2$, where $\si =[R]\# \si'$,
then the result obviously follows.
\item If $[\a]\m \b. M_1)\ra^{R} M_2\fl^{\si'}M_3$, where $M_2\neq N$ and
$\m \b. M_1$ does not disappear in $\si$, then $M_3=[\a]\m \b. M_3'$ and $M_1\fl^{\si}M_3'$,
 which yields the result.
\item If $[\a]\m \b. M_1 \ra^{R} M_2\fl^{\si'}M_3$, where $M_2\neq N$ and
$\m \b. M_1$ disappears in $\si$. Then $[\a]\m
\b. M_1 \fl^{\si''} [\a]\m \b.[\b]M_k \ra_{\th} [\a]M_k  \fl^{\si'''}M_3$, where $\m \b. M_1$
does not disappear in $\si''$. We have $[\a]\m \b.M_1 \ra_{\rho}M_1[\b :=\a ]\fl^{\si''[\b :=\a ]}([\b]M_k)[\b :=\a
]= [\a]M_k \fl^{\si'''}M_3$, and the latter reduction sequence is
equal in length to $\si$. By this the result follows.
\end{enumerate}
The reverse direction is obvious.

\item Similar to the above one. \qedhere
\end{enumerate}
\end{proof}

\begin{lem} \label{ch2:lgi}\label{ch2:etm}
Let $M=(\mu \a. M_1)M_2\ora{P}$ and $N=(\mu \a. M_1[\a
:=_rM_2])\ora{P'}$.
\begin{enumerate}
\item If $\a \in Fv(M_1)$ and $N$ is strongly normalizable, then $M$ is also strongly normalizable $\eta (M)=\eta(N)+1$.
\item If $\a \notin Fv(M_1)$ and $N,M_2$ are strongly normalizable, then $M$ is also strongly normalizable and $\eta (M)=
\eta (N)+\eta (M_2)+1$.
\end{enumerate}
\end{lem}

\begin{proof}
\begin{enumerate}
\item Let $M\ra^{\si}M^*$. We prove $|\si| \leq \eta (N)+1$, from this
$\eta (M)\leq \eta (N)+1$ follows.
\begin{enumerate}
\item The redex $R=(\m \a. M_1)M_2$ is involved in $\si$.
\begin{enumerate}
\item If $\m \a. M_1$ does not disappear in $\si$,
$(\mu \a. M_1)M_2\ora{P}\ra^{\si'}(\mu \a. M_1')M_2'\ora{P'}
\ra_{\m}(\mu \a. M_1'[\a:=_rM_2'])\ora{P'}\fl^{\si''}M^*$. Then,
since $\a \in Fv(M_1)$, by Lemmas \ref{ch2:cmp} and
\ref{ch2:sbtrm}, the reduction sequence $(\mu
\a. M_1)M_2\ora{P} \ra^r(\mu \a. M_1[\a
:=_rM_2])\ora{P} \fl (\mu \a. M_1[\a :=_rM_2'])\ora{P} \fl (\mu \a.
M_1'[\a:=_rM_2'])\ora{P'} \fl^{\si''}M^*$ has length at least
$|\si |$, by which the assertion follows.
\item If $\m \a. M_1$ disappears in $\si$,
$(\mu \a. M_1)M_2\ora{P} \ra^{\si'}(\mu \a.
[\a]M_1')M_2'\ora{P'}\ra_{\th}$\\
$(M_1'\;M_2'\ora{P'})\fl^{\si''}M^*$. Then, since $\a
\in Fv(M_1)$, by Lemmas \ref{ch2:cmp} and \ref{ch2:sbtrm}, the
sequence $(\mu \a. M_1)M_2\ora{P} \ra_{\m}(\mu
\a. M_1[\a :=_rM_2])\ora{P} \fl (\m \a.([\a]M_1')[\a :=_rM_2'])\ora{P'}
=(\m \a.[\a](M_1'\;M_2'))\ora{P'})
\ra_{\th}(M_1'\;M_2'\ora{P'})\fl^{\si''}
M^*$ has length at least $|\si|$, which yields the result.
\end{enumerate}
\item The redex $R=(\m \a. M_1)M_2$ is not involved in $\si$.
\begin{enumerate}
\item If $\m \a. M_1$ does not disappear in $\si$, that is, $(\mu \a. M_1)M_2\ora{P} \fl M^*=$\\
$(\m \a. M_1')M_2'\ora{P'}$. Then, since $\a \in
Fv(M_1)$, we can apply Lemmas \ref{ch2:cmp} and \ref{ch2:sbtrm} to
assert that $(\mu \a. M_1)M_2\ora{P} \ra^R(\mu \a. M_1[\a
:=_rM_2])\ora{P} \fl (\mu \a. M_1'[\a :=_rM_2'])\ora{P}'$ has
length at least $|\si |+1$.
\item If $\m \a. M_1$ disappears in $\si$,
$(\mu \a. M_1) M_2\ora{P} \fl (\m \a.[\a]M_1')M_2'\ora{P'} \ra_{\th}$\\
$ (M_1'\;M_2'\ora{P'})\fl
M^*$. By Lemmas \ref{ch2:sbtrm} and \ref{ch2:cmp}, the sequence\\
$(\mu \a. M_1)M_2\ora{P} \ra^R (\mu \a. M_1[\a
:=_rM_2])\ora{P} \fl (\m
\a. ([\a]M_1')[\a :=_rM_2'])\ora{P'} =(\m
\a.[\a](M_1'\;M_2'))\ora{P'} \ra_{\th}(M_1'\;M_2' \; \ora{P'})\fl
M^*$ has length at least $|\si|+1$, which proves the assertion.
\end{enumerate}
\end{enumerate}
The reverse direction is obvious.

\item The proof of $\eta (M)\leq \eta (N)+\eta (M_2)+1$ is similar to
the first part of the proof of Lemma \ref{ch2:lgi}. In this case
the verification is made easier by the fact that, since $\a \notin
Fv(M_1)$, $\m \a. M_1$ does not disappear
in a reduction sequence starting from $M$.
For the converse, let $N\fl^{\si}N'$ and $M_2\fl^{\nu}M_2'$.
Then $(\m \a. M_1) M_2\ora{P} \ra^{\nu}(\m \a.
M_1)M_2'\ora{P} \ra_{\m}(\m \a. M_1)\ora{P} \ra^{\si}N'$ is a
reduction sequence starting from $M$, which means that $\eta
(N)+\eta (M_2)+1\leq \eta(M)$. \qedhere
\end{enumerate}
\end{proof}

\subsection{The general case}

In what follows we transform every $\l \m$-term $M$ into a $\l
\m I$-term $[\![M]\!]_k$ with some $k\geq 0$ such that $\eta (M)\leq \eta ([\![M]\!]_k)$,
by which, using Corollary \ref{ch2:bfi}, we can obtain a bound for
$\eta (M)$.

At this point our presentation slightly differs from that of Xi \cite{Xi}. We have reformulated the translation in \cite{Xi},
hence we were able to avoid the minor mistake of Xi when computing the complexity of the obtained $\l\m$I-terms.
For a detailed explanation see \cite{Batt}.
The interesting fact for Theorem \ref{ch2:bflmr}, which is the main result of the paper, is, however,
that we get the same bound for the simply typed $\l \m$-calculus as Xi obtained for the $\l$-calculus,
mutatis mutandis. Namely, if we restrict the notion of the rank of a term in Definition \ref{ch1:lhtype}
by taking into consideration the $\b$-redex only, we get the result of Xi for the $\l$-calculus as
a special case of Theorem \ref{ch2:bflmr}. This suggests that the addition of the classical variables,
together with the new rules, does not increase the computational complexity of the calculus.
The idea of the translation is to introduce new variables of appropriate types in order to ensure
that each bounded variable appears in the terms. The only difficult case is that of the $\l$-abstraction.
We explain below the difficulties lying behind the definition for the case of the abstraction.

\begin{defi}\label{ch2:tr_lm_lmi}
\hfill
\begin{enumerate}
\item Let $\ma{V}=\{v_{(A,B)}\;|\;A, B\textrm{ are types}\}$ be a set
of distinguished variables such that for all $A$, $B$ we have
$v_{(A,B)}:A\ra(B\ra A)$, where $v_{(A,B)}$ are either constants or new variables.
Let $M:A$ and $N:B$ be typed $\l \m$-terms. We denote the term $((v_{(A,B)}\;M)\;N)$ by $\lan M ,
N\ran$.
\item Let $M$ be a term and $k\geq 0$. The $\l \m$-term $[\![M]\!]_k$ assigned to $M$ is defined as
follows.
\begin{itemize}
\item $[\![M]\!]_k=M$, if $M$ is a variable,
\item $[\![M]\!]_k=\l x.\l y_1.\dots \l y_m.\langle ([\![M_1]\!]_k\;y_1\dots
y_m) , x\rangle$, if $M=\l x.M_1$ such that $lh(type(M))\leq k$ and
$type(M_1)=A_1\ra \dots \ra A_m\ra B$, $type(y_i)=A_i$ $(1\leq
i\leq m)$ and $B$ is atomic,
\item $[\![M]\!]_k=\l x.\langle [\![M_1]\!]_k , x\rangle$, if $M=\l x.M_1$ and
$lh(type(M))> k$,
\item $[\![M]\!]_k=\m \a.\langle [\![M_1]\!]_k , [\a] z\rangle$, if $M=\m \a. M_1$,
where $\a \notin Fv(M_1)$ and $z$ is a new variable such that
$type(M)=type(z)$,
\item $[\![M]\!]_k=\mu \a. [\![M_1]\!]_k$, if $M=\m \a. M_1$ and $\a \in
Fv(M_1)$,
\item $[\![M]\!]_k=[\a][\![M_1]\!]_k$, if $M=[\a]M_1$,
\item $[\![M]\!]_k=([\![M_1]\!]_k\;[\![M_2]\!]_k)$, if $M=(M_1\;M_2)$.
\end{itemize}
\item For each term $M$ and each $k \geq 0$, we define the contexte $\G_{M,k}$ wich containes
the constants $v_{(A,B)}$ of $[\![M]\!]_k$ with their type $A\ra(B\ra A)$.
\end{enumerate}
\end{defi}

Observe that in the definition above the translations for $\l$- and $\m$-abstractions differ.
The underlying reason is the fact that in a $\b$-reduction the $\l$-abstraction disappears while
this is not the case concerning a $\m$-reduction. Hence, in order to ensure the validity of Lemma \ref{ch2:lhtr},
we must make sure that the translation of a term with $\b$-redex as head redex can be continued even after
the reduction with the head redex. The main aim with the translation is to produce a $\lm$I-term $[\![M]\!]_k$ from
$M$ such that the relation $\eta(M)\leq \eta([\![M]\!]_k)$ should be valid, which is the statement of Lemma \ref{ch2:lhtr}.
To achieve this, we reproduce the original $M$ inside its translation $[\![M]\!]_k$ in a sense, since, in general,
the translation does not respect reduction, that is, if $M\ra N$, then it is not necessarily the case that $[\![M]\!]_k\ra [\![N]\!]_k$.\\

The next four lemmas describe some intuitively clear properties of the translation.

\begin{lem} \label{ch2:lIvar}
 Let $M$ be a term and $k\geq 0$.
\begin{enumerate}
\item $[\![M]\!]_k$ is a $\l \m I$-term.
\item $\a \in Fv(M)$ iff $\a \in Fv([\![M]\!]_k)$ and
if $x \in Fv(M)$, then $x \in Fv([\![M]\!]_k)$.
\end{enumerate}
\end{lem}

\begin{proof}
By induction on $comp(M)$.
\end{proof}

Observe that, in the case of $\l$-variables, $Fv(M)\subseteq Fv([\![M]\!]_k)$, since $M$ can contain free variables of
the form $v_{(A,B)}$ besides its original parameters.

\begin{lem} \label{ch2:rankeq}
If $M$ is a term and $k \geq 0$, then $rank([\![M]\!]_k)=rank(M)$.
\end{lem}

\begin{proof}
By induction on $comp(M)$.
\end{proof}

\begin{lem}\label{ch2:sizetr}
If $M$ be a term and $k \geq 0$, then $comp([\![M]\!]_k)\leq (2k+3)\cdot comp(M)$.
\end{lem}

\begin{proof}
The only nontrivial case is $M=\l x.M_1$. Let $lh(type(\l
xM_1))=l$. If  $k<l$, then $comp([\![M]\!]_k)=comp(\l x.\lan [\![M_1]\!]_k , x\ran )=comp([\![M_1]\!]_k)+3
\leq (2k+3)\cdot comp(M)$.
If $k\geq l$, then, for some $m\leq l$, we obtain by the induction hypothesis
$comp([\![M]\!]_k) = comp(\l x.\l y_1.\dots \l y_m.\lan ([\![M_1]\!]_k\;y_1\dots y_m) , x\ran ) = comp([\![M_1]\!]_k)+2m+3\leq
(2k+3)\cdot comp(M)$.
\end{proof}

\begin{lem} \label{ch2:sbth}
Let $M,N$ be terms and $k\geq 0$.
\begin{enumerate}
\item If $\G \v M : A$, then $\G , \G_{M,k} \v [\![M]\!]_k : A$.
\item $[\![M]\!]_k[x:=[\![N]\!]_k]=[\![M[x:=N]]\!]_k$,
\item $[\![M]\!]_k[\a :=_r[\![N]\!]_k]=[\![M[\a :=_rN]]\!]_k$,
\item $[\![M]\!]_k[\b :=\a ]=[\![M[\b :=\a ]]\!]_k$.
\end{enumerate}
\end{lem}

\begin{proof}
By induction on $comp(M)$.
\end{proof}

Our aim is to prove $\eta (M)\leq \eta ([\![M]\!]_k)$. The assertions of Subsection \ref{ch2:subseceta} and Lemma
\ref{ch2:etrd} prepare the proof of that statement, which is the claim of Lemma \ref{ch2:lhtr}.

\begin{lem} \label{ch2:etrd}
If $M\ra M'$ and $rank(M)\leq k$, then $\eta ([\![M']\!]_k)+1\leq \eta ([\![M]\!]_k)$.
\end{lem}

\begin{proof}
By induction on $comp(M)$.
\begin{enumerate}
\item If $M=\l x.M_1$, the induction hypothesis applies.
\item If $M=(\l x.M_1) M_2\dots M_n$. We have $lh(type(\l x.M_1))\leq k$ by virtue of the assumption $rank(M)\leq k$.
Let $type(M_1)=A_1\ra \dots A_m\ra B$, where $B$ is atomic. Let
$M'=(M_1[x:=M_2]\dots M_n)$, otherwise the induction hypothesis
applies. Since $B$ is atomic, $m\geq n-2$ holds. Then
$[\![M]\!]_k \ra \l y_1.\dots \l y_m.\langle ([\![M_1]\!]_k[x:=[\![M_2]\!]_k]\;y_1\dots
y_m),[\![M_2]\!]_k \rangle \dots [\![M_n]\!]_k$\\
$\fl \l y_{n-1}.\dots \l y_m.\langle ([\![M_1]\!]_k[x:=[\![M_2]\!]_k]\dots
[\![M_n]\!]_k\;y_{n-1}\dots y_m),[\![M_2]\!]_k\rangle$.
Lemma \ref{ch2:sbth} gives
$([\![M_1]\!]_k[x:=[\![M_2]\!]_k]\;[\![M_3]\!]_k\dots [\![M_n]\!]_k)=
([\![M_1[x:=M_2]]\!]_k\;[\![M_3]\!]_k\dots [\![M_n]\!]_k) =$\\
$([\![M_1[x:=M_2]\;M_3\dots M_n]\!]_k)=[\![M']\!]_k$, by which the result follows.
\item If $M=(\mu \a. M_1) M_2\dots M_n$, we may assume again that
$M\ra^{R}M'$, where $R=(\m \a. M_1) M_2$.
\begin{enumerate}
\item If $\a \in Fv(M_1)$, let $M'=(\mu \a. M_1[\a :=_rM_2]\dots M_n)$.
We have, by Lemma \ref{ch2:sbth},
$[\![M]\!]_k = (\mu \a. [\![M_1]\!]_k)[\![M_2]\!]_k\dots [\![M_n]\!]_k \;\ra\;(\mu
\a. [\![M_1]\!]_k[\a :=_r[\![M_2]\!]_k]\dots [\![M_n]\!]_k) = (\mu \a. [\![M_1[\a
:=_rM_2]]\!]_k\dots [\![M_n]\!]_k)\;=\;[\![M']\!]_k$.
\item If $\a \notin Fv(M_1)$, then $M'=(\m \a. M_1)M_3\dots M_n$ and $[\![M]\!]_k \ra$\\
$(\m \a. \langle [\![M_1]\!]_k, [\a](z\;[\![M_2]\!]_k)\rangle)[\![M_3]\!]_k\dots [\![M_n]\!]_k$.
We may assume $\a\notin Fv(M_i)$ $(1\leq i\leq k)$. Then
Lemma \ref{ch2:lhm} gives \\
$\eta ((\m \a. \langle [\![M_1]\!]_k, [\a](z\;[\![M_2]\!]_k)\rangle)[\![M_3]\!]_k\dots [\![M_n]\!]_k)\geq$ \\
$\eta((\m \a. \langle [\![M_1]\!]_k,
[\a]z\rangle)[\![M_3]\!]_k\dots [\![M_n]\!]_k)+\eta([\![M_2]\!]_k)+1$.
Moreover, by induction on $n$, we obtain that $\eta ([\![M']\!]_k)\leq$
$\eta ((\m \a. \langle [\![M_1]\!]_k,
[\a]z\rangle)[\![M_3]\!]_k\dots [\![M_n]\!]_k)$, by which the result
follows.
\end{enumerate}
\item If $M=[\a]M_1$, the only interesting case is $M=[\a]\m \b. M_1' \ra M_1'[\b :=\a
]$. If $\b \in Fv(M_1')$, then $[\![M]\!]_k=(\a\;\m \b. [\![M_1']\!]_k)$.
Otherwise, $[\![M]\!]_k=[\a]\m \b.\lan [\![M_1']\!]_k ,[\b]z\ran$.
Applying Lemma \ref{ch2:sbth}, in both cases we obtain the result.
\item The case $M=\m \a. M_1$ is analogous to the previous one.
\item If $M=(x\;M_1\overrightarrow{P})$, the induction hypothesis applies. \qedhere
\end{enumerate}
\end{proof}

Prior to proving the next lemma, we demonstrate with an example that the hypothesis $rank(M)\leq k$ was
indeed necessary for the validity of Lemma \ref{ch2:etrd}.

\begin{exa}
Let $M=(\l x.\l y.y\;x)\;y$. Then $M'=(\l y.y)y$.
Assume $x,y:A$. Then $rank(\lambda x.\lambda y.y)=2$, which means $rank(M)=2$. Let $k=1$.
Then

$[\![M]\!]_1=((\l x.\lan [\![\l y.y]\!]_1,x\ran) x)y$, and
$[\![M']\!]_1=(\l y.\lan y,y\ran)y$.

Since $[\![\l y.y]\!]_1=\l y.\lan y,y\ran$ is not a redex, we have $\eta([\![M]\!]_1)=\eta([\![M']\!]_1)=1$,
thus the statement of Lemma \ref{ch2:etrd} is not valid for $M$.
\end{exa}

\begin{lem} \label{ch2:lhtr}
If $M$ is a $\l \mu$-term such that $rank(M)\leq k$, then $\eta
(M)\leq \eta ([\![M]\!]_k)$.
\end{lem}

\begin{proof}
By induction on $( \eta ([\![M]\!]_k),comp(M))$.
\begin{enumerate}
\item If $M=\l x.M_1$, then, by the induction hypothesis we have the result.
\item If $M=(x\;M_1\dots M_n)$, then, by the induction hypothesis,
$\eta (M)=\eta (M_1)+\dots +\eta (M_n)\leq \eta ([\![M_1]\!]_k)+\dots +
\eta ([\![M_n]\!]_k)=\eta ([\![M]\!]_k)$.
\item If $M=(\l x.M_1) M_2\dots M_n$, let $M'=(M_1[x:=M_2]\dots
M_n)$. If follows from Lemma \ref{ch2:rank} that $rank(M')\leq
k$.
\begin{itemize}
\item[-] $x\in Fv(M_1)$: By Lemmas \ref{ch2:etrd} and \ref{ch2:etr} and the
induction hypothesis,
$\eta (M)= \eta (M')+1\leq \eta ([\![M']\!]_k)+1\leq \eta ([\![M]\!]_k)$.
\item[-] $x\notin Fv(M_1)$: $[\![M]\!]_k\ra_{\beta}\l y_1.\dots \l y_m.\langle ([\![M_1]\!]_k\;y_1\dots
y_m) , [\![M_2]\!]_k\rangle\ldots [\![M_n]\!]_k=U$. By Lemma \ref{ch2:etr2}, we are ready, if we prove $\eta([\![(M_1\;M_3\ldots M_n)]\!]_k\leq \eta(U)$.
By the choice of $m$, we have $m\geq n-2$, hence\\
$U\fl \langle ([\![M_1]\!]_k\;[\![M_3]\!]_k\dots
[\![M_m]\!]_k) , [\![M_2]\!]_k\rangle\ldots [\![M_n]\!]_k$, from which the conclusion follows.
\end{itemize}
\item Let $M=(\mu \a. M_1)M_2\dots M_n$.
\begin{itemize}
\item[-] If $\a \in Fv(M_1)$, let $M'=(\mu \a. M_1[\a :=_rM_2]\dots M_n)$.
Then $rank(M')\leq k$ by
Lemma \ref{ch2:rank} again. We have, by Lemmas
\ref{ch2:etrd}, \ref{ch2:lgi} and the induction hypothesis,
$\eta (M)=\eta (M')+1\leq \eta ([\![M']\!]_k)+1=\eta ([\![M]\!]_k)$.
\item[-] If $\a \notin Fv(M_1)$, let $M'=(\m \a. M_1)M_3\dots M_n$.
We have\\
$[\![M]\!]_k \ra  (\m \a. \langle [\![M_1]\!]_k,
[\a](z\;[\![M_2]\!]_k)\rangle)[\![M_3]\!]_k\dots
[\![M_n]\!]_k$, which, together with Lemmas
\ref{ch2:etm}, \ref{ch2:lhm}, \ref{ch2:rank} and the induction hypothesis, yields that
$\eta (M)\leq \eta (M')+\eta (M_2)+1\leq \eta ([\![M']\!]_k)+\eta ([\![M_2]\!]_k)+1\leq \eta ([\![M]\!]_k)$.
\end{itemize}
\item Let $M=\m \a. M_1$.
\begin{itemize}
\item[-] Assume $\a \in Fv(M_1)$. If $\mu \a. M_1=\m \a. [\a]M_2$ is a
$\th$-redex, then, by Lemmas \ref{ch2:eth}, \ref{ch2:rank} and
the induction hypothesis, $\eta (M)=\eta(M_2)+1\leq
\eta([\![M_2]\!]_k)+1=\eta([\![M]\!]_k)$. Otherwise, let $\m \a. M_1\ra M'$.
Since $\m \a. M_1$ is not a $\th$-redex, we have $M'=\m \a. M_1'$
together with $rank(M')\leq k$. By Lemma \ref{ch2:etrd}, we can
apply the induction hypothesis to $M'$, that is,
$\eta (\m \a. M_1')+1\leq \eta ([\![\m \a. M_1']\!]_k)+1\leq \eta ([\![\m \a. M_1]\!]_k)$.
But $M'$ was arbitrary and $\eta (M)=max\{\eta (M')+1\;|\;M\ra
M'\}$, which proves our assertion.
\item[-] If $\a \notin Fv(M_1)$, then we can apply the induction hypothesis to $M_1$.
\end{itemize}
\item Let $M=[\a]\m \b. M'$. Similar to the previous case by using Lemma \ref{ch2:erh}. \qedhere
\end{enumerate}
\end{proof}

The following theorem is the main result of our paper. Interestingly, as mentioned before, we obtain the same
bound for the $\lmrt$-calculus as that for the $\l$-calculus \cite{Xi}.

\begin{thm} \label{ch2:bflmr}
If $M$ is a $\l \m$-term such that $rank(M)=k$, then every
$\bmrt$-reduction sequence starting from $M$ is of length less than
$tower(k+1,(2k+3)\cdot comp(M))$.
\end{thm}

\begin{proof}
We obtain, by Lemma \ref{ch2:sizetr}, $comp([\![M]\!]_k)\leq (2k+3)\cdot
comp(M)$ and, by Lemma \ref{ch2:rankeq}, $rank([\![M]\!]_k)=rank(M)$.
These, together with Corollary
\ref{ch2:bfi} and Lemma \ref{ch2:lhtr}, imply
$\eta (M) \leq  \eta ([\![M]\!]_k) < tower(k+1,comp([\![M]\!]_k))
\leq tower(k+1,(2k+3)\cdot  comp(M))$.
\end{proof}

\section{Concluding remarks}\label{final:section}

In what follows, we give a short account of the other possibilities for obtaining bounds for the reduction sequences in the $\lm$-calculus.
We could have also begun our paper with these considerations, however the methods below do not give such full-fledged results as the one
discussed above (the bounds are higher and, more importantly, we were unable to treat the additional rules by the arguments presented below).
By this reason, we decided to deal with these discussions only after the main argument of the paper. We could resort to the idea of
translating the $\lm$-calculus into the $\l$-calculus by a CPS-translation such that the sizes of the translated terms and the lengths
of their reduction sequences would depend on the sizes and lengths of the original terms. Then the bound for the $\l$-calculus would
provide us with a bound for the $\lm$-calculus, too. By examining this idea, we have come to the conclusion that we were not able to simulate every
reduction rule, if we apply the already existing translations, and even the
bound would be much worse than the one appearing in our result. We investigate these questions in detail below.

\subsection{A possible attempt to compute an upper bound for the \texorpdfstring{$\lmr$}{lambda-mu-rho}-calculus}

In the following observations we confine our attention to the case
of the $\lmr$-calculus. In order to establish a bound for the lengths
of reduction sequences of the $\lmr$-calculus it seems to be a
natural idea to try to transform a reduction sequence of the
$\lmr$-calculus into a reduction sequence of the $\l$-calculus. We
go round this approach a little bit more detailed: we present the
CPS-translation from the simply-typed $\lmr$-calculus to the
simply-typed $\l$-calculus introduced by de Groote \cite{de Gro2},
and then we give an account of the possibilities of finding an
appropriate bound with this method. The notation for the
CPS-translation is taken from de Groote
\cite{de Gro2}. As to a bound for the simply-typed $\l$-calculus we
regard the one presented in Xi \cite{Xi}.

\begin{defi}
Let $o$ be some distinguished atomic type.
\begin{enumerate}
\item For every type $A$, we define the three types
$A^o$, $\sim A$ and $\overline{A}$ by : $\sim A=A\rightarrow o$, $\overline{A}=\sim
\sim A^o$, ${\perp}^o=o$, $X^o=X$, if $X$ is atomic, and $(B\rightarrow C)^o=\overline{B}\rightarrow \overline{C}$.
\item  Let $\G_{\l}$ (resp. $\G_{\mu}$) denote a $\l$-context (resp. $\m$-context),
that is, a finite (possibly empty) set of declarations
$x_1:A_1,\dots,x_n:A_n$ (resp. $\a_1:\neg B_1,\dots,\a_m : \neg B_m$).
We define $\overline{\G_{\l}}$ (resp. $\sim \G_{\m} ^o$) by
$x_1:\overline{A_1},\dots,x_n:\overline{A_n}$ (resp. $\a_1:\sim B_1^o,\dots,\a_m : \sim B_m^o$).
\end{enumerate}
\end{defi}

We suppose that the $\mu$ -variables of the $\lmr$-calculus
are also $\l$ -variables of the $\l$-calculus.

\begin{defi}
The CPS-translation $\up{M}$ of a $\lmr$-term $M$ is defined as
follows.
\begin{itemize}
\item $\up{x}=\l k.(x \;k)$,
\item $\up{\l x.M}=\l k.(k \; \l x.\up{M})$,
\item $\up{(M\;N)}=\l k.(\up{M} \;\l m. (m \; \up{N}\;k))$,
\item $\up{\mu \a. M}=\l \a. (\up{M} \;\l k. k)$,
\item $\up{[\a] M}=\l k.(\up{M} \;\a)$.
\end{itemize}
\end{defi}

\begin{lem} Let $M:A$ be a typable term with
$\l$-context $\G_{\l}$ and $\m$-context $\G_{\m}$. Then its
CPS-translation, $\overline{M}$, is typable with contexts
$\overline{\G_{\l}}$ and $\sim \G_{\m}^o$.
\end{lem}

\begin{defi}
Let $=_{\l}$ (resp. $=_{\m}$) denote the relation defined as the
reflexive, symmetric, transitive closure of the $\b$-reduction
(resp. that of the union of the $\b$-, $\m$- and
$\rho$-reductions). As usual, we consider terms differing in
renaming of bound variables as equals.
\end{defi}

Then, in \cite{de Gro3}, de Groote proves the following result.

\begin{lem} \label{ch2:eq}
Let $M,N$ be $\lmr$-terms. Then $M=_{\mu}N$ iff $\overline{M}=_{\l}\up{N}$.
\end{lem}

Unfortunately, in Lemma \ref{ch2:eq},
$\overline{M}\fl_{\l}\up{N}$ does not hold in general, even if
$M\fl_{\mu}N$. So on one hand we cannot use the CPS-translation
to imitate the reduction sequences in the $\lmr$-calculus by
reduction sequences in the $\l$-calculus. On
the other hand there can be another drawback of this approach.\\
In general, we could make use of the CPS-translation for
estimating bounds of reduction sequences if, for any
$M\fl_{\mu}^{\si}NF(M)$, we could find a $\nu$ with
$\up{M}\fl_{\l}^{\nu} NF(\up{M})$ such that $|\si |\leq c.
|\nu |$ with some constant $c$, where $NF(M)$ and $NF(\up{M})$
denote the (unique) normal form of $M$ in the $\lmr$-calculus and of
$\up{M}$ in the $\l$-calculus, respectively. In fact, we even
know that $NF(\up{M})=\up{\up{NF(M)}}$, where $\up{\up{M}}$ stands
for the so called modified CPS-translation of the term $M$ ( de
Groote \cite{de Gro3}).

For the moment suppose for every reduction sequence
$M\fl^{\si}NF(M)$ we can find a reduction sequence $\nu$ such
that $\up{M}\fl_{\l}^{\nu} NF(\up{M})$ with $|\si |\leq c\cdot
|\nu |$. By the result for the $\b$-normalization in Xi
\cite{Xi}, we would have for any $\nu$ as above\\
$|\nu |<c\cdot  \textrm{tower}(rank(\up{M})+1,(2\cdot  rank(\up{M})+3)\cdot  comp(\up{M}))$.\\

On the other hand we have the following estimations.

\begin{lem}
Let $M$ be a $\lm$-term. Then $rank(\up{M})=3\cdot  rank(M)$ and \\
$2\cdot  comp(M)<comp(\up{M})$.
\end{lem}

This means that the best estimation for the lengths of the
reductions with this method would be greater than $c\cdot
\textrm{tower}(3\cdot  rank(M)+1,(12\cdot  rank(M)+6)\cdot  comp(M))$, and
by the direct method this upper bound is
$\textrm{tower}(rank(M)+1,(2\cdot  rank(M)+3)\cdot  comp(M))$. At
present, no CPS-translation which could yield a significantly
better estimation is known to the authors.

\subsection{A translation of the \texorpdfstring{$\lm$}{lambda-mu}-calculus into the \texorpdfstring{$\l_c^*$}{lambda-c-star}-calculus}

Some years ago David and Nour \cite{Dav-Nou5} discovered a translation of the $\lm$-calculus into the $\l$-calculus with recursive
equations for types. This is somewhat simpler than the
CPS-translation and provides an easy method for finding an
estimation for the lengths of the reduction sequences in the
$\lm$-calculus. We present a version of their translation establishing a
connection between the $\lm$-calculus and a variant of the
$\l$-calculus enlarged with some constants.  The method traces back to
Krivine \cite{Kri1,Kri2}, where he supplemented the typed calculus with
a constant of type $\forall X(\neg \neg X\ra X)$.

\begin{defi}
Enhance the set of types of the simply typed $\l$-calculus with an element $\bot$ and
define $\neg A$ as $A\ra \bot$. Let $X$ be an atomic type, add for
each $X$ a new constant $c_X$ of type $\neg \neg X\ra X$. Let us
call the new calculus as $\l_c^*$.
We define for each type $A$ a closed $\l_c^*$-term $T_A$ such
that $T_A$ has the type $\neg \neg A\ra A$.
\begin{itemize}
\item $T_{\bot}=\l y.(y \; I)$, where $I=\l x.x$,
\item $T_X=c_X$, where $X$ is atomic,
\item $T_{A\ra B}=\l x.\l y.(T_B\;\l z.(x\;\l t.(z\;(t\;y))))$.
\end{itemize}
\end{defi}

We suppose that the $\mu$ -variables of the $\lm$-calculus
are also $\l$ -variables of the $\l$-calculus.

\begin{defi}
 Let $k\geq 0$. We define a $k$-translation $\|.\|_k$ of the set of $\lm$-terms into the
set of terms of the $\l_c^*$-calculus as follows.
\begin{itemize}
\item $\|x\|_k=x$,
\item $\|\l x.M\|_k=\l x.\|M\|_k$,
\item $\|(M\;N)\|_k=(\|M\|_k\;\|N\|_k)$,
\item $\|\m \a. M\|_k=(T_A \;\l \a.\|M\|_k)$, if $\a$ has type $\neg A$ and $lh(A)\leq k$,
\item $\|\m \a. M\|_k=(z \; \|M\|_k)$, if $\a$ has type $\neg A$ and $lh(A)>k$ and where
$z:\bot \ra A$ is a new variable,
\item $\|[\a]M\|_k=(\a\;\|M\|_k)$.
\end{itemize}
\end{defi}

In the above definition the $\m$-variables and its translated
counterparts were denoted with the same letters. Let $\v_{\lm}$
and $\v_{\l_c^*}$ denote the typing relations in the $\lm$- and in
the $\l^*$-calculus, respectively. We have the following
assertions.

\begin{lem} \label{ch2:trtypepres}
Let $k\geq 0$ and $M$ a $\lm$-term. If $\G\v_{\lm}M:A$, then $\G\v_{\l_c^*}\|M\|_k:A$.
\end{lem}

\begin{proof}
Straightforward.
\end{proof}

\begin{lem} \label{ch2:trest}
Let $k\geq 0$, $M$, $N$ be $\lm$-terms such that $rank(M) \leq k$.\\
If $M \ra_{\lm}N$, then $\|M\|_k\fl^+_{\l} \|N\|_k$.
\end{lem}

\begin{proof}
Obviously, it is enough to check the relation $\|(\m \a. M_1)M_2\|_k\fl^+_{\l}$\\
$\|\m \a. M_1[\a :=_rM_2]\|_k$, where, necessarily,
$k\geq lh(A)$ provided $type(\a)=\neg A$.
\end{proof}

\begin{lem} \label{ch2:trestm}
Let $k\geq 0$, $M$, $N$ be $\lm$-terms such that $rank(M) \leq k$. \\
If $M\fl^nN$, then $\|M\|_k\fl^m \|N\|_k$ for some $m\geq n$.
\end{lem}

\begin{proof}
Follows from Lemmas \ref{ch2:rank} and \ref{ch2:trest}.
\end{proof}

Since no reduction rules are added to $\l$ when defining $\l_c^*$,
the method of Xi \cite{Xi} for estimating the lengths of reduction
sequences is also applicable to $\l_c^*$ without any changes. We
state without proof the following theorem.

\begin{thm} \label{ch2:lhl*}
Let $M$ be a $\l_c^*$-term such that $rank(M)=k$. Then every
reduction sequence starting from $M$ has length less than
$\textrm{tower}(k+1, (2k+3)\cdot comp(M))$.
\et

In order to establish a bound for the lengths of $\lm$-reduction
sequences we have to estimate the size of the translated terms as
well.

\begin{lem} \label{ch2:trtypesize}
If $A$ is a type, then $comp(T_A)\geq 8\cdot lh(A)+3$.
\end{lem}

\begin{proof}
Obvious.
\end{proof}

\begin{lem} \label{ch2:trsize}
If $M$ is a $\lm$-term such that $rank(M)=k$, then
$comp(\|M\|_k)\leq (8k+4)\cdot comp(M)$.
\end{lem}

\begin{proof}
By induction on $comp(M)$. We only check one of the cases. Let
$M=(\m \a. M_1)M_2$. Assume $type(\a)=\neg A$. Then, since $k\geq
lh(A)$, we have, by Lemma \ref{ch2:trtypesize} and the induction
hypothesis, $comp(\|M\|_k) = comp(\|\m \a. M_1\|_k)\|M_2\|_k$\\
$ = comp((T_A\;\l \a.\|M_1\|_k))+comp(\|M_2\|_k)
 \leq (8k+4)+comp(\|M_1\|_k)+comp(\|M_2\|_k)
 \leq (8k+4)\cdot comp(M)$.
\end{proof}

\begin{thm}
Let $M$ be a $\lm$-term such that $rank(M)=k$. Then every reduction
sequence starting from $M$ has length less than\\
$\textrm{tower}(k+1, (2k+3)\cdot (8k+4)\cdot comp(M))$.
\end{thm}

\begin{proof}
Follows from Theorem \ref{ch2:lhl*} and Lemma \ref{ch2:trsize}.
\end{proof}

This method, however, is not applicable to the $\lmrt$-calculus,
since, in the cases of the $\rho$- and $\theta$-reductions, Lemma \ref{ch2:trest}
is not valid.

\section{Future work}

In this paper, we have shown how to find a bound for the lengths of the reduction sequences in the simply typed $\l\m$-calculus extended
with the rules $\rho$ and $\theta$.
The bound depends only on the size of the term and on the maximum of the ranks of its redexes.
We first gave a bound concerning $\l\m$I-terms, then we established a correspondence between $\lm$-terms
and $\l\m$I-terms such that the lengths of the longest reduction sequences do not decrease. To obtain a bound for the
$\lm$I-calculus we defined the notion of a standard $\bmrt$-reduction sequence, the formulation of which was not entirely
straightforward because of the presence of
overlapping redexes. Surprisingly, we have obtained that, with the necessary changes, the same bound is appropriate for
the $\lm$-calculus as the one found by Xi \cite{Xi} for the $\l$-calculus \cite{Xi}. This leads us to the conjecture that
the computational complexity of the $\l$-calculus is not enhanced by the introduction of the classical variables and the new rules.
Our future work can be the following.

\begin{enumerate}
\item {\bf Finding a term realizing the bound.} In the literature usually different upper bounds can be found for the lengths of the
reduction sequences in the $\l$-calculus. A question naturally arises, which bounds could be the most precise ones? Could we amend the
present bounds considerably?

\item {\bf Commutation lemmas for the $\rho$- and $\theta$-rules.} If we considered only the $\b$- and $\mu$-rules,
our proof would simplify considerably, especially the arguments concerning standardization. However, the question arises whether,
in the cases of the $\rho$- and the $\theta$-rules, we are able to prove commutation lemmas together with maintaining an upper bound
for the lengths of the reductions. It would be good to see whether this approach simplifies the presentation of our results or not.

\item{\bf Other rules for the $\l\m$-calculus.} The $\l\m$-calculus can be considered with other kinds of reductions.
For example, one can prohibit two consecutive $\m$-abstractions ($\m\a.\m\b. M$) or $\m$-variable
applications ($[\a][\b]M$) (see Nour \cite{Nou1}). Parigot has also proposed a rule which prohibits that a
$\l$-abstraction should immediately follow a $\m$-variable (Py \cite{Py}). Moreover, in Saurin's paper \cite{Sau}, there are some additional rules:
Saurin proves a standardization theorem with respect to his calculus. Another rule is also worth considering:
$(N \: \m \a. M)\ra_{{\mu}'} \m \a. M[\a :{=}_{l}N]$, where $M[\a :{=}_{l}N]$ is obtained from $M$ by
replacing every subterm in $M$ of the form $[\a]U$ by $[\a](N \, U)$. This rule is the symmetric
counterpart of the $\m$-rule, the addition of which makes the $\l\m$-calculus non-confluent,
but the strong normalization still holds \cite{Dav-Nou3,Dav-Nou4}. A standardization result in relation to the $\lmp$-calculus
is obtained by David and Nour in \cite{Dav-Nou2}. Concerning our results, we think that the same bound could also be obtained for the $\lmp$-calculus.
Presumably, the proof would be a little more involved than the one presented in this article, however, we do not intend to elaborate it.

\item {\bf Other classical calculi.} It would be interesting to find an upper bound for the lengths of the reduction sequences
in other classical calculi \cite{Ber-Bar,Cur-Her}. The question naturally arises whether the methods presented in this paper are applicable to them.

\end{enumerate}

\bigskip

\section*{Acknowledgment}
\noindent  We wish to thank Ren\'e David for helpful discussions.

\end{document}